\documentclass{elsarticle}
\usepackage[T1]{fontenc}
\usepackage{booktabs}
\usepackage{url}
\usepackage{graphicx}
\usepackage{comment}
\usepackage{amsmath,amsfonts,amssymb,amsthm}
\usepackage{dsfont}
\usepackage{xcolor}
\usepackage{placeins}
\usepackage{paralist}
\usepackage{verbatim}

\newcommand{\ind}{\mathds{1}}

\theoremstyle{definition}
\newtheorem{definition}{Definition}

\theoremstyle{proposition}
\newtheorem{proposition}{Proposition}

\theoremstyle{remark}

\theoremstyle{theorem}
\newtheorem{theorem}{Theorem}

\newcommand{\rect}{\mathrm{rect}}
\newcommand{\orth}{\mathrm{orth}}

\newcommand{\ext}{\mathrm{ext}}
\newcommand{\per}{\mathrm{per}}
\newcommand{\ctr}{\mathrm{ctr}}
\newcommand{\sym}{\mathrm{sym}}
\newcommand{\cad}{\mathrm{cad}}
\newcommand{\asd}{\mathrm{asd}}
\newcommand{\mix}{\mathrm{mix}}

\newcommand{\mc}{\mathrm{MC}}

\newcommand{\bsone}{\boldsymbol{1}}
\newcommand{\bszero}{\boldsymbol{0}}
\newcommand{\bsgamma}{\boldsymbol{\gamma}}
\newcommand{\bseta}{\boldsymbol{\eta}}
\newcommand{\bsdelta}{\boldsymbol{\Delta}}
\newcommand{\bsx}{\boldsymbol{x}}
\newcommand{\bsa}{\boldsymbol{a}}
\newcommand{\bsb}{\boldsymbol{b}}
\newcommand{\bsv}{\boldsymbol{v}}
\newcommand{\bsc}{\boldsymbol{c}}
\newcommand{\bsy}{\boldsymbol{y}}
\newcommand{\bsz}{\boldsymbol{z}}
\newcommand{\rd}{\mathrm{\, d}}

\newcommand{\e}{\mathbb{E}}

\begin{document}
\begin{frontmatter}
\title{On the optimization of discrepancy measures}

\author[1]{Fran\c{c}ois Cl\'{e}ment}
\ead{fclement@uw.edu}
\author[2]{Nathan Kirk\corref{cor1}}\ead{nmk7@st-andrews.ac.uk}
\author[3]{Art B. Owen}\ead{owen@stanford.edu}
\author[4,5,6]{T. Konstantin Rusch}\ead{tkrusch@tue.ellis.eu}

\cortext[cor1]{Corresponding author}

\affiliation[1]{organization={Department of Mathematics, University of Washington}}
\affiliation[2]{organization={School of Mathematics and Statistics, University of St Andrews}}
\affiliation[3]{organization={Department of Statistics, Stanford University}}
\affiliation[4]{organization={ELLIS Institute T\"{u}bingen}}
\affiliation[5]{organization={Max Planck Institute for Intelligent Systems}}
\affiliation[6]{organization={T\"{u}bingen AI Center}}

\date{}

\begin{abstract}
Numerical optimizations can find point sets with lower discrepancies than prior methods based on algebra, and $L_2$ discrepancies are computationally attractive for this. The $L_2$ star discrepancy unreasonably emphasizes the origin and Matou\v{s}ek pointed to a serious drawback. 
He showed that the $L_2$ star discrepancy of $n$ points at $(1,\dots,1)\in[0,1]^d$ is less than the root mean squared $L_2$ star discrepancy of $n$ random points until $n>(3/2)^d-1$.
We introduce the \emph{average squared discrepancy} (ASD) which averages $2^d$ versions of the squared $L_2$ star discrepancy, one for each vertex of $[0,1]^d$. We show that it can be computed in $O(dn^2)$ work and that it does not suffer from Matou\v{s}ek's repeated-point pathology.  We inspect several $L_2$ discrepancies that do not favor the origin over other vertices.  Only two of them avoid the repeated-point pathology: a symmetric discrepancy and a periodic discrepancy. The symmetric discrepancy has a twin-point pathology where point sets with only $2$ distinct points can be better than random points until $n\ge 3^d$. We also study weighted $L_2$ discrepancies introduced by Hickernell that use some product weights. Weighting lower dimensional projections reduces but does not eliminate the repeated-point pathologies. The ASD is equivalent to the weighted symmetric $L_2$ discrepancy upon taking all weights equal to $4$.  It is subject to a recursive twin-point pathology.
A range of numerical experiments showcase recent progress in the construction of sets with low $L_2$ discrepancies for many different measures. In our examples, optimizing the ASD also gives good results for \emph{the other} discrepancies we consider, especially in lower dimensions.
\end{abstract}

\end{frontmatter}

\section{Introduction}

Placing $n$ points in $[0,1]^d$ to minimize measures of discrepancy, i.e., placing them as uniformly as possible,
has been a fundamental challenge since the early 20th century \cite{vand:1935:I,vand:1935:II,weyl:1914,weyl:1916}.
The resulting low-discrepancy point sets are most commonly used in multivariate
integration \cite{DicEtal14a,HICKKIRKSOR2025,  practicalqmc}, but have applications in computer graphics \cite{kell96}, neural network training \cite{mishra21} and path planning in robotics \cite{chahine24} to name a few examples.

Low-discrepancy point sets and sequences have typically been constructed using number theoretic properties \cite{lattice_rules:2022,dick:pill:2014,joe:kuo:2008,nied92,sloa:joe:1994} where consequently the best results are obtained for special values of $n$ such
as powers of $2$ or large prime numbers, and there has been a strong emphasis on asymptotic convergence rates as $n\to\infty$. Multiple approaches have also been suggested to obtain constructions for fixed $n$ and $d$. Older approaches relied on $\delta$-covers~\cite{DGKP08,DGW10} to guarantee a discrepancy bound but were not useful in practice as they barely, if at all, outperformed the asymptotic constructions such as the Sobol' sequence~\cite{Sob67}. More recently, computationally intense optimization techniques have grown in importance~\cite{cle22,CDKP,Rainville,rusc:etal:2024,van2025neural}. These methods have greatly lowered empirical discrepancy bounds, in particular in low dimensions.

Given the multitude of discrepancy criteria~\cite{chen:sriv:trav:2014}, these new optimization developments raise the question of which discrepancy should be the target of the optimization.
The $L_\infty$ star discrepancy, denoted $D_\infty^*$, is perhaps the most widely studied, largely due to the Koksma–Hlawka inequality \cite{hlaw:1961,PAUSINGER2015} which bounds the integration error by a product of the variation of the integrand and the $L_{\infty}$ star discrepancy of the point set. 
However, it is extremely difficult to compute in practice. The best current algorithm \cite{DEM96} costs $O(n^{d/2+1})$ as $n\to\infty$ for fixed $d$. In a setting where $d$ also grows,  the problem is NP-hard \cite{complexity}.

In addition to its high cost, $D^*_{\infty}$ is ill-suited to gradient-based optimization algorithms. Although it is Lipschitz continuous, hence differentiable almost everywhere, finding the gradient requires us to find a maximizing anchored box.
 
We can avoid the costs of $D_\infty^*$ by using the $L_2$ star discrepancy, $D_2^*$. Unlike $D_\infty^*$, $D_2^*$ has a closed-form due to Warnock~\cite{warn:1972} that we give as equation~\eqref{eq:warnock}. The Warnock formula allows $D_2^*$ to be computed in $O(dn^2)$ time.  It is differentiable almost everywhere with a simple formula for the derivative.  
Furthermore, 
it can be used to bound integration error through $L_2$ analogues of the Koksma-Hlawka inequality \cite{FraHei96,hickdisc,NovWoz09,NovWoz10a}.

Both $D^*_\infty$ and $D_2^*$ are defined through anchored boxes, that is, hyper-rectangu\-lar subsets of $[0,1]^d$ which include the origin. Emphasizing the origin over the other $2^d - 1$ vertices in $\{0,1\}^d$ simplifies the analysis of such point sets, but is not always well suited to a given integration problem. Indeed, as each parameter range is typically rescaled to $[0,1]$, anchoring at the origin implicitly suggests that minimal parameter values are more important, an assumption which is not generally appropriate. That issue motivated an alternative $L_2$ discrepancy averaging over all axis-parallel hyperrectangles \cite{moro:cafl:1993}, which we include here as the `extreme' $L_2$ discrepancy.

A second problem with $D_2^*$ is a repeated-point pathology noted by Matoušek in~\cite{mato:1998:2}. He shows there that placing all $n$ points at $\bsone = (1,1,\dots,1)$ will yield a value of $D_2^*$ that is smaller than the root mean square (RMS) value of $D_2^*$ for $n$ independent and identically distributed (IID) points from $U[0,1]^d$ unless $n>(3/2)^d-1$. When we say that a pathological point set is better than random points, it is always with this comparison to the RMS value.

As we show here, the extreme $L_2$ discrepancy is subject to an even worse version of the repeated-point pathology.  Placing all $n$ points at any single boundary point of $[0,1]^d$ including $\bsone$
gives a lower extreme $L_2$ discrepancy than the RMS value of random points until $n>2^d-1$. The repeated-point pathology for $D_2^*$  was not known at the time that the extreme $L_2$ discrepancy was proposed.

\subsection{Contributions}
A good discrepancy for numerical optimization should allow efficient gradient computations, should not treat the origin specially, and should not be subject to a repeated-point pathology.
We introduce the \emph{average squared discrepancy} (ASD), denoted by $D_2^\asd$. The ASD averages $2^d$ versions of $D_2^*$, one for each corner of $[0,1]^d$, so it does not favor the origin over other vertices. It is an $L_2$ analog of the \emph{multiple-corner} discrepancy recently introduced in~\cite{CDKP}.
A naive implementation of ASD would raise the computational cost by a factor of $2^d$. We present a Warnock-style formula for it at cost $O(dn^2)$ in Theorem \ref{thm:itsodnsq}.  We show that the ASD does not have a repeated-point pathology.

We also study several geometrically motivated $L_2$ discrepancy measures that unlike $D_2^*$, do not favor the origin over other vertices.  The periodic and symmetric discrepancies have no repeated-point pathology while extreme, centered, mixture and center-anchored discrepancies are vulnerable to it. The repeated-point in these pathologies is not always $\bsone$.

In addition to the repeated-point pathology, we have found some other pathologies.  The symmetric discrepancy has a twin-point pathology involving $n$ points that come as $n/2$ pairs. In the most extreme case a configuration with only two distinct points is better than random points until $n\ge3^d$. More generally, any point set made of vertices with equally many even and odd vertices, or of pairs of points that differ only in one coordinate switched between $0$ and $1$, attains the same value. The pathologies for the star discrepancy, extreme discrepancy and a centered discrepancy extend to much more general configurations than repeated points. For the star discrepancy we could confine all $n$ points to be in a union of $d$ hyperplanes, for the extreme discrepancy they could all be within one of $2d$ hyperplanes and for the centered discrepancy they could all be within one of $d$ hyperplanes. These latter three cases have pathologies as bad as their respective repeated-point pathologies.

The only discrepancy that we have not found a pathology for is a periodic $L_2$ discrepancy. We don't have a way to define all possible pathologies.  The ones we have discovered involve a point set $P$ of $n$ points, and an axis-parallel box $B$ of volume $1/2$ or greater where $P$ places no points in $B$ and yet is better than $n$ random points until $n>\lambda^d$ for some $\lambda>1$.  The pathologies we present are the worst ones we have found; some others might be worse.

All seven discrepancies can be generalized to a weighted family of $L_2$ discrepancies as Hickernell \cite{hickdisc,hick1998} showed for most of them, and the resulting squared discrepancy is then a weighted sum of $2^d-1$ squared discrepancies for coordinate projections of the points.  Including lower dimensional projections greatly reduces the intensity of repeated-point pathologies but does not eliminate them.

Using the Warnock formula we show that the ASD is a constant multiple of a squared weighted symmetric discrepancy of~\cite{hick1998}.  Setting all the weights equal to $4$ makes the criterion equivalent to ASD. This ASD connection adds a direct interpretation of that weighted discrepancy.

Because ASD sums over coordinate projections of the squared symmetric discrepancy it is less susceptible than the symmetric discrepancy to the twin-point pathology. We exhibit a recursive twin-point pathology for it where sample sizes are powers of $2$ and it requires $n\gtrsim 1.223^d$ to overcome. The points in that construction resemble a certain classical experimental design, while having some coordinate projections that are singletons.

Finally, we provide a wide range of numerical experiments on several of 
the discrepancy measures described in this paper. The first aim is to provide a benchmark for future optimization attempts as we find numerical results on $L_2$ discrepancies of any kind are surprisingly rare in the literature, much more so than for $D^*_\infty$. The second is to showcase what optimization techniques such as those introduced in~\cite{rusc:etal:2024} are capable of, using the Sobol' sequence as a reference. The last contribution is to run a cross-evaluation where we compare the quality of these points sets for measures \emph{on which they were not optimized}. 
In our examples, optimizing for the ASD leads to strong results for all other $L_2$ discrepancies, particularly in low dimensions. For example, optimizing for the ASD gave nearly as good results for $D_2^*$ as optimizing for $D_2^*$ directly did. The highest dimension we considered was $d=10$ and a pathology at or near $n=1.223^{10}\approx7.5$ does not affect the higher sample sizes. We also show that it is possible to use these optimized sets as starting points for inexpensive sequential constructions. For different measures such as the extreme or average squared discrepancies, we use an approach similar to that of the greedy $D^*_2$ construction in~\cite{clem:2023,krit:2022} to obtain a sequence over a few hundred points that maintains low-discrepancy throughout.

\subsection{Paper Overview}
Section~\ref{sec:discandopt} has notation and some general comments on $L_2$ discrepancies.  Section~\ref{sec:l2discreps} describes the different $L_2$ discrepancies that we study, starting with $D_2^*$ and then listing others that do not give special consideration to the origin of $[0,1]^d$. For each, we give a geometric definition that we take as the fully $d$-dimensional discrepancy. In all but one case we give a formula for the fully $d$-dimensional discrepancy as well as for a weighted $L_2$ generalization. Those formulas are currently spread out through the literature. We also exhibit several pathologies there. Section~\ref{sec:pathoweights} summarizes the pathologies of our $L_2$ discrepancies and shows how weighted repeated-point pathologies are less severe than fully $d$-dimensional ones. 
Section~\ref{sec:averagesquared} introduces the average squared discrepancy, and shows that it avoids Matoušek's repeated-point pathology and can be computed at $O(dn^2)$ cost. We also show there that it is equivalent to the weighted symmetric discrepancy \cite{hick1998} upon setting all the weights equal to four. We exhibit a pathology for it that is milder than the one for $D_2^\sym$.
Section~\ref{sec:numerical} presents numerical results comparing the optimization of the different discrepancy measures. Section~\ref{sec:mpmc} contains the main numerical results, focusing on the use of the Message-Passing Monte Carlo (MPMC) technique \cite{rusc:etal:2024} to generate low $L_2$ discrepancy sets. In our examples, the ASD was usually better than the periodic discrepancy (which also avoids the pathology) when we did a cross-evaluation. Section~\ref{sec:graddesc} gives an alternative technique for practitioners with very little computing power. Finally, Section~\ref{sec:greedy} presents a greedy sequential construction approach, inspired by Kritzinger's~\cite{krit:2022} construction for the $L_2$ star discrepancy in one dimension.

\section{Notation and general remarks on $L_2$ discrepancies}\label{sec:discandopt}

Here we make some general remarks about
the $L_2$ discrepancies we study. We assume familiarity with $D_2^*$ which is the RMS discrepancy of a point set over anchored boxes $[\bszero,\bsa)$ for $\bsa\sim U[0,1)^d$. Here we provide some notation, describe weighted $L_2$ discrepancies in general and give an expression for the mean squared discrepancies under Monte Carlo sampling. The star discrepancy and other $L_2$ discrepancies are then studied in Section~\ref{sec:l2discreps}.

\subsection{Notation}
For integers $n,d\ge1$, we study various discrepancy measures for points $\bsx_1,\dots,\bsx_n\in[0,1]^d$. We refer to these points as the point set $P$. We take $P$ to be a multi-set, allowing duplicated entries. Indeed the pathologies we study often involve repeated points. For a measurable set $A\subset[0,1]^d$, the local discrepancy of $P$ in $A$ is
\[
\delta(A) = \delta(A;P) = \frac1n\sum_{i=1}^n\ind_{\{\bsx_i\in A\}}-|A|
\] 
where $|\cdot|$ is Lebesgue measure. The components of $\bsx_i$ are denoted $x_{i,k}$ for $k\in1{:}d:=\{1,2,\dots,d\}$.

By Monte Carlo (MC) points we mean $n$ independent points $\bsx_i\sim U[0,1]^d$. For MC points $\e[\delta(A)]=0$ and $\e[\delta(A)^2]=\tau^2(A)/n$ for
some $\tau^2(A)\in[0,1/4]$. We use $P_\mc$ to denote an MC point set, leaving $n$ implicit.

We work with many $L_2$ discrepancies defined as the RMS value of the local discrepancy with respect to a distribution on sets $A$. For a point set $P$, a generic geometrically motivated $L_2$ discrepancy is denoted $D_2^\bullet(P)$. All our discrepancies give $\e[ (D_2^\bullet(P_\mc))^2]=\e[\tau^2(A)]/n$ with the final expectation being with respect to the distribution of $A$. When we say that some pathological construction $P$ is better than random points we mean that $(D_2^\bullet(P))^2<\e[( D_2^\bullet(P_\mc))^2]$.

We also study some weighted $L_2$ discrepancies. For $u\subseteq1{:}d$, we let $P_u$ denote the coordinate projection of the points in $P$, retaining only their $k$th components for $k\in u$. For a vector $\bsgamma\in(0,\infty)^d$ of weights, the generic weighted $L_2$ discrepancy is defined by
\[
(D_2^{\bullet,\bsgamma}(P))^2 = \sum_{u\ne\emptyset}D_2^\bullet(P_u)^2\prod_{k\in u}\gamma_k.
\]
Some authors allow $\gamma_k=0$ \cite{sloa:wozn:1997}. Then any projection involving $x_k$ contributes nothing to the discrepancy. Our convention ensures that the fully $d$-dimensional criterion $D_2^\bullet(P)$ is included in its weighted generalization. 

Most of the discrepancies and weighted discrepancies in this section were devised by Hickernell \cite{hickdisc,hick1998} well before the arrival of powerful methods to optimize them. Some discrepancies first appeared as a sum over $2^d-1$ projections, equivalent to setting all $\gamma_k=1$. Those generalize to any $\bsgamma\in(0,\infty)^d$ and then the fully $d$-dimensional version $(D_2^\bullet(P))^2$ can be recovered as the coefficient of $\prod_{k=1}^d\gamma_k$ in $(D_2^{\bullet,\bsgamma}(P))^2$, modulo a mild complication that we will see for a symmetric discrepancy.

\subsection{Special points and sets}
The center point of $[0,1]^d$ is $\bsc=(1/2,1/2,\dots,1/2)$. A vertex of $[0,1]^d$ is any point $\bsv\in\{0,1\}^d$. A boundary point of $[0,1]^d$ is any point $\bsz$ with at least one coordinate $z_k\in\{0,1\}$. That is $\bsz\in\partial[0,1]^d$.

To remove the special status of the origin, we consider larger families of sets, including boxes not anchored at the origin, and certain unions of boxes. Each family of test sets $A$ admits a well-defined geometric interpretation.

Many of our discrepancy measures are defined in terms of  rectangular hulls.  For $a,b\in[0,1]$ we let $\rect(a,b)=\bigl[\min(a,b),\max(a,b)\bigr)$. Then the rectangular hull of $\bsa,\bsb\in[0,1]^d$ is
\[\rect(\bsa, \bsb) = 
\prod_{k=1}^d\rect(a_k,b_k).\]
We will also denote the nearest vertex to $\bsa\in[0,1]^d$ by $\bsv(\bsa)$.
Defining 
\[
v_k(\bsa) = \ind_{\{a_k \ge 1/2\}}
\]
ensures uniqueness when $a_k = 1/2$ for some $k\in1{:}d$.

The general $L_2$ discrepancy measures we consider do not depend on whether
the rectangles we use are open, closed or half-open
because we are averaging over the locations of the points (such as $\bsa$) used to define their boundaries.  When using methods from the literature our definitions
might not match the open-ness or closed-ness from
the way those methods were originally presented. 

\subsection{Monte Carlo performance}

The (unweighted) squared discrepancies we consider take the form
 \[
 A - \frac{2}{n} \sum_{i=1}^n \prod_{k=1}^d B(x_{i,k}) + \frac{1}{n^2} \sum_{i,j=1}^n \prod_{k=1}^d C(x_{i,k}, x_{j,k}),
 \]
 for a constant $A$ and functions $B$ and $C$.
Then \[
\e[(D_2^\bullet(P_\mc))^2]=
A-2\e[B(u)]^d+\Bigl(1-\frac1n\Bigr)\e[C(u,v)]^d+\frac1n\e[C(u,u)]^d
\]
for independent \( u, v \sim U[0,1] \). Because this mean squared discrepancy is proportional to $1/n$ we have
\begin{align}\label{eq:dsquaredbullet}
\e[(D_2^\bullet(P_\mc))^2]=
\frac1n\Bigl(\e[C(u,u)]^d-\e[C(u,v)]^d\Bigr).
\end{align} 

\subsection{Reproducing kernel Hilbert spaces}

A formal connection between discrepancy theory and RKHSs has been known since \cite{FraHei96,hick:1996:weights}. Although we do not detail this approach here, it is important to note that the RKHS framework defines discrepancies via ANOVA-type decompositions that include marginal effects. Those are the contributions from projections of the point set onto lower-dimensional subsets of coordinates.

The relationship between discrepancies that consider only fully $d$-dimensional effects and those incorporating marginal contributions can be understood in terms of the symmetric and positive definite kernels associated with the RKHS. Specifically, if \( \prod_{k=1}^d \tilde{K}(x_k, y_k) \) denotes the product kernel corresponding to a discrepancy that evaluates only the full \( d \)-dimensional distribution (excluding projections), then the analogous kernel incorporating all marginal effects is $K(\bsx, \bsy) = \prod_{k=1}^d \bigl( 1 + \gamma_k \tilde{K}(x_k, y_k) \bigr)$
where \( \bsgamma = (\gamma_1, \ldots, \gamma_d)\in(0,\infty)^d\). 

\section{Specific $L_2$ discrepancies}\label{sec:l2discreps}

Here we give details on $D_2^*$ and the other $L_2$ discrepancies. Our first subsection, about $D_2^*$, includes some general remarks. 

\subsection{Star discrepancies}\label{sec:stardiscrepancy}

In the early 20th century, research primarily focused on the uniform placement of points in $[0,1]^d$ with respect to unanchored axis-parallel boxes. The Weyl criterion for equidistribution \cite{weyl:1914,weyl:1916} is a prominent example of this. While it is difficult to pinpoint the exact point in the literature, 
a shift occurred around the mid-20th century towards a discrepancy which considered axis-parallel boxes anchored at the origin. This was later referred to as the \emph{star discrepancy}. This shift in focus was likely motivated by practical considerations; the star discrepancy is both computationally and conceptually simpler than the extreme discrepancy while maintaining the same asymptotic behavior in $n$.

\begin{definition}[$L_{\infty}$ star discrepancy] The $L_\infty$ \emph{star discrepancy} of a point set $P$ is  
\begin{equation}
    D^*_{\infty}(P)
    =
    \sup_{\bsa\in[0,1]^d} |\delta([\bszero,\bsa))|,
\end{equation}
where $[\bszero,\bsa)=\{\bsx\in[0,1]^d\mid 0\le x_k< a_k, 1\le k\le d\}$
is the anchored box at $\bsa$.
\end{definition}

As noted above, computations for an $L_2$ version of star discrepancy are much more suitable than $D_\infty^*$ for optimization algorithms. Here and elsewhere we present squared $L_2$ discrepancies to simplify some expressions.

\begin{definition}[$L_2$ star discrepancy]
The squared \emph{$L_2$ star discrepancy} of a point set $P$ is
\begin{equation}
    (D_2^*(P))^2
    = \int_{[0,1]^d}\delta([\bszero,\bsa))^2\rd\bsa.
\end{equation}
\end{definition}
It is the RMS local discrepancy with respect to $\bsa\sim U[0,1]^d$. Warnock's formula \cite{warn:1972} is
\begin{align}\label{eq:warnock}
(D_2^*(P))^2 = \frac1{3^d}-\frac2n\sum_{i=1}^n\prod_{k=1}^d\frac{1-x_{i,k}^2}2
+\frac1{n^2}\sum_{i,j=1}^n\prod_{k=1}^d \left[1- \max\bigl(x_{i,k},x_{j,k}\bigr) \right].
\end{align}
The case with $d=1$ first appeared in~\cite{zare:1968}.

We see from equation~\eqref{eq:warnock} that $D_2^*$ is continuous in $x_{i,k}$. It can be computed in $O(dn^2)$ time, making it much more practical than $D^*_\infty$. There is an algorithm for $D_2^*$ with cost $O(n(\log n)^{d-1})$ \cite{hein:1996} but the simplicity of equation~\eqref{eq:warnock} makes it the usual choice. The gradient of $D_2^*$ exists almost everywhere. By computing and reusing the appropriate leave-one-coordinate-out products in Warnock's formula, the entire gradient can be evaluated in $O(dn^2)$ time.

For large $d$, most anchored boxes have a very small volume. Then placing a point $\bsx_i$ near the origin creates a large discrepancy, as it will be counted for all these small boxes. This issue underlies Matoušek's repeated-point pathology, which we can quickly verify. If $P$ is $n$ copies of $\bsone$ then $(D_2^*(P))^2=3^{-d}$. Letting $C(u,v)=1-\max(u,v)$ we get $\e[(D_2^*(P_\mc))^2]$$=(2^{-d}-3^{-d})/n$ by
\eqref{eq:dsquaredbullet}. MC points are only better when $n>(3/2)^d-1$.

To show that a discrepancy has the repeated-point pathology we only need to find one point $\bsz$ for which $n$ copies of $\bsz$ is better than $P_\mc$ until $n>\lambda^d$ for some $\lambda>1$. This $\bsz$ does not have to be the point that minimizes a repeated-point discrepancy. Indeed, for $d=2$, taking $n$ copies of the point $(t,t)$ for $t=(\sqrt{5}-1)/2$ gives a smaller $D_2^*$ than taking $n$ copies of $\bsone$ does.  To show that a discrepancy does \emph{not} suffer from the pathology, we do need to compare $n$ random points to any $P$ made of $n$ repeated points.

The repeated-point pathology is a special case of a more general pathology.  Let every point $\bsx_i$ have $\max_{1\le k\le d}x_{i,k}=1$. Then $(D_2^*(P))^2=3^{-d}$. So the pathological point set can be scattered over $d$ upper faces of $[0,1]^d$, leaving $[0,1)^d$ empty. As a result $D_2^*$ has many pathological target point sets $P$ for an optimization algorithm to find.

A weighted generalization of $D_2^*(P)$ is given by equation (30c) of \cite{hick1998}. It is
\begin{multline}\label{eq:wtdl2star}
    (D_2^{*,\bsgamma}(P))^2
    = \sum_{u\ne\emptyset} D_2^*(P_u)^2\prod_{k\in u}\gamma_k\\
    =
    \prod_{k=1}^d\left(1+\frac{\gamma_k}{3}\right)-\frac{2}{n}\sum_{i=1}^n\prod_{k=1}^d\left(1+\frac{\gamma_k}2(1-x_{i,k}^2)\right)\\+\frac{1}{n^2}\sum_{i,j=1}^{n}\prod_{k=1}^d \left(1+\gamma_k(1-\max(x_{i,k},x_{j,k}))\right).
    \end{multline}

Our computations use an equal weight version taking all $\gamma_k=1$:
\begin{multline}\label{eq:l2starones}
(D_2^{*,\bsone}(P))^2=\left(\frac{4}{3}\right)^d-\frac{2}{n}\sum_{i=1}^n\prod_{k=1}^d\frac{3-x_{i,k}^2}2\\
+\frac{1}{n^2}\sum_{i,j=1}^{n}\prod_{k=1}^d (2-\max(x_{i,k},x_{j,k})).
\end{multline}
Equation~\eqref{eq:l2starones} appeared in \cite{hickdisc} before the general version~\eqref{eq:wtdl2star} was published.

Equation~\eqref{eq:l2starones}
corresponds exactly to the Warnock formula, after adding 1 to each factor in every term. We return to this relation for the other discrepancies at the end of Section~\ref{sec:mix}.

\subsection{The extreme discrepancy}
The $L_\infty$ discrepancy taken over all axis-aligned
boxes $[\bsa,\bsb)\subset[0,1]^d$ is called the extreme discrepancy. Some older papers, such as~\cite{warn:1972,Zar69}, use ``extreme'' discrepancy to refer to $L_{\infty}$ discrepancies as opposed to $L_2$ discrepancies. Here we give an $L_2$ extreme discrepancy.

\begin{definition}[Extreme discrepancy]
The squared $L_2$ \emph{extreme discrepancy} of a point set $P$ is
\begin{align}\label{eq:extl2disc}
(D_2^{\ext}(P))^2 = \int_{[0,1]^d} \int_{[0,1]^d} \ind_{\{\bsa \le \bsb\}} \, \delta([\bsa,\bsb))^2 \rd\bsa \rd\bsb,
\end{align}
where $\bsa \le \bsb$ holds component-wise. 
\end{definition}

This discrepancy was used in~\cite{moro:cafl:1993} to de-emphasize the origin.  They showed that
\begin{multline}\label{eq:ext}
(D_2^{\ext}(P))^2=\frac1{12^d}-\frac2{n} 
\sum_{i=1}^n \prod_{k=1}^d \frac{x_{i,k}(1-x_{i,k})}{2}\\
+\frac1{n^2}\sum_{i,j=1}^n\prod_{k=1}^d\bigl( \min(x_{i,k},x_{j,k})-x_{i,k}x_{j,k}\bigr).
\end{multline}

A weighted version of $D_2^\ext$ appears in equation (33b) of~\cite{hick1998}. For $\bsgamma\in(0,\infty)^d$, define
\begin{multline}\label{eq:extwt}
(D_2^{\ext,\bsgamma}(P))^2=
\sum_{u\ne\emptyset}D_2^\ext(P_u)^2\prod_{k\in u}\gamma_k\\
=\prod_{k=1}^d\left(1+\frac{\gamma_k}{12}\right)-\frac{2}{n}\sum_{i=1}^n\prod_{k=1}^{d}\left(1+\frac{\gamma_kx_{i,k}(1-x_{i,k})}{2}\right)\\
+\frac{1}{n^2}\sum_{i,j=1}^n\prod_{k=1}^d \left(1+\gamma_k(\min(x_{i,k},x_{j,k})-x_{i,k}x_{j,k})\right).
\end{multline}
To verify the generalization, we see that $(D_2^\ext(P))^2$ from equation~\eqref{eq:ext} is the coefficient of $\prod_{k=1}^d\gamma_k$ in~\eqref{eq:extwt}. 
Our $\gamma_k$ corresponds to $\hat\beta_k^2$ of~\cite{hick1998}.

For $n\ge1$ identical points, $(D_2^{\ext})^2=12^{-d}+(1-2^{-d+1})\prod_{k=1}^dx_{1,k}(1-x_{1,k})$.
If $d=1$, this is $12^{-d}$ for any $\bsx_1$. For $d\ge2$, $(D_2^\ext)^2$ attains its minimum value of $12^{-d}$ for $n$ copies of any point $\bsz\in\partial[0,1]^d$ and at no interior point $\bsz$. By equation~\eqref{eq:dsquaredbullet}, the expected value of $(D_2^\ext(P_\mc))^2$ is $(6^{-d}-12^{-d})/n$. This is better than  $n$ identical boundary points when $n>2^d-1$, so this pathology is even worse than the one for $D_2^*$.

This pathology extends beyond repeated boundary points.  If $x_{i,k}=0$ for some $k\in1{:}d$ and all $i=1,\dots,n$, then $(D_2^\ext(P))^2=12^{-d}$. The same holds if $x_{i,k}=1$ for some $k$ and all $i$. Confining $P$ to any of $2d$ boundary hyperplanes leaves $(D_2^\ext(P))^2=12^{-d}$ while leaving $(0,1)^d$ with no sample points.

\subsection{The periodic discrepancy}

A \emph{periodic discrepancy} of \cite{Lev95} 
was studied in \cite{hickdisc} for use on periodic functions.
It is also known as \emph{wrap-around discrepancy} or \emph{Weyl  discrepancy}. This discrepancy treats $[0,1)^d$ as a torus. For $a, b \in [0,1]$, define the wrapped interval
\[
W(a,b) = 
\begin{cases}
[a,b), & a \le b, \\
[0,b) \cup [a,1), & a > b,
\end{cases}
\]
and extend the definition to $\bsa, \bsb \in [0,1]^d$ by
\(
W(\bsa,\bsb) = \prod_{k=1}^d W(a_k, b_k).
\)
\begin{definition}[Periodic discrepancy]
The $L_2$ periodic discrepancy of a point set $P$ is
\begin{align}\label{eq:perl2disc}
(D_2^{\per}(P))^2 = \int_{[0,1]^d} \int_{[0,1]^d} \delta(W(\bsa, \bsb))^2 \rd\bsa \rd\bsb.
\end{align}
\end{definition}

If for some $\bsdelta\in[0,1)^d$ we shift every point $\bsx_i$ to $\{\bsx_i+\bsdelta\}$ where $\{\cdot\}$  denotes the fractional part applied componentwise, then $D_2^\per$ is unchanged, as shown by  \cite{Lev95}.
This periodic discrepancy (and its weighted version below) are the only shift invariant discrepancies we study here.
The test sets used for the periodic discrepancy include all the boxes used for $D_2^\ext$ and $D^*_2$.  

Equation~(34b) of \cite{hick1998} has
\begin{multline}\label{eq:perwt}
(D_2^{\per,\bsgamma}(P))^2=-\prod_{k=1}^d\left(1+\frac{\gamma_k}{3}\right)\\
+\frac{1}{n^2}\sum_{i,j=1}^n\prod_{k=1}^d \left(1+\gamma_k\left(\frac12-|x_{i,k}-x_{j,k}|+(x_{i,k}-x_{j,k})^2\right)\right).
\end{multline}
More recent papers~\cite{hinr:krit:pill:2020,hinr:oett:2016} work with the following fully $d$-dimensional formula
\begin{align}\label{eq:per}
(D_2^{\per}(P))^2=-\frac1{3^d}
+\frac1{n^2}\sum_{i,j=1}^n\prod_{k=1}^d \Bigl( \frac12 - |x_{i,k}-x_{j,k}|+(x_{i,k}-x_{j,k})^2\Bigr),
\end{align}
which is the coefficient of $\prod_{k=1}^d\gamma_k$ in $(D_2^{\per,\bsgamma}(P))^2$.

By invariance it does not matter where $n$ identical points are placed. For any $n$ identical points, $(D_2^\per)^2 = 2^{-d}-3^{-d}$. By equation~\eqref{eq:dsquaredbullet}, $\e[(D_2^\per(P_\mc))^2]$$=(2^{-d}-3^{-d})/n$, which is better than identical points when $n>1$. Therefore $D_2^\per$ does not incur a repeated-point pathology. Since $D_2^\per$ avoids the pathology, so does $D_2^\per(P_u)$ for all $u\ne\emptyset$. It follows that $D^{\per,\bsgamma}_2$ avoids the pathology for all $\bsgamma\in(0,\infty)^d$.
We have not found any pathology for $D_2^\per$. 

\subsection{Centered discrepancy}\label{sec:ctr}
In the \emph{centered discrepancy} \cite[Section 5.2]{hickdisc} the test sets are axis-aligned boxes connecting each point $\bsa \in [0,1]^d$ to $\bsv(\bsa)$, its closest vertex of the cube. 
\begin{definition}[Centered discrepancy]
The $L_2$ centered discrepancy of a point set $P$ is
\begin{align}\label{eq:ctrl2disc}
(D_2^{\ctr}(P))^2 = \int_{[0,1]^d} \delta(\rect(\bsa, \bsv(\bsa)))^2 \rd\bsa.
\end{align}
\end{definition}

A squared centered discrepancy appears in Section 5.2 of \cite{hickdisc}.
The weighted version from equation (31c) of \cite{hick1998} is
\begin{multline}\label{eq:weightedctr}
(D_2^{\ctr,\bsgamma}(P))^2 = \sum_{u\ne\emptyset} (D_2^\ctr(P_u))^2\prod_{k\in u}\gamma_k\\=
\prod_{k=1}^d \left( 1 + \frac{\gamma_k}{12} \right) -
\frac{2}{n} \sum_{i=1}^n \prod_{k=1}^d
\left[ 1 + \frac{\gamma_k}{2} \left( \left| x_{i,k} - 1/2 \right| - \left( x_{i,k} - 1/2 \right)^2 \right) \right] \\
+ \frac{1}{n^2} \sum_{i,j=1}^n
\prod_{k=1}^d \left[ 1 + \frac{\gamma_k}{2} \left( \left| x_{i,k} - 1/2 \right| + \left| x_{j,k} - 1/2 \right| - \left| x_{i,k} - x_{j,k} \right| \right) \right].
\end{multline}
We can extract the fully $d$-dimensional part
\begin{multline}\label{eq:dctr2}
     (D_2^{\ctr}(P))^2 = \frac{1}{12^d}
	- \frac 2n \sum_{i=1}^{n} \prod_{k=1}^d \frac 12 \bigl ( \lvert x_{i,k} - 1/2 \rvert - (x_{i,k} - 1/2 )^2 \bigr )  \\
	 \qquad \qquad +\frac{1}{n^2} \sum_{i,j=1}^{n} \prod_{k = 1}^d  \frac 12 \bigl ( \lvert x_{i,k} - 1/2 \rvert + \lvert x_{j,k} - 1/2 \rvert - \lvert x_{i,k} - x_{j,k} \rvert \bigr)
\end{multline}
as the coefficient of $\prod_{k=1}^d\gamma_k$ in \eqref{eq:weightedctr}.
The original unit-weighted squared centered discrepancy from~\cite{hickdisc} can be obtained via $(D_2^{\ctr,\bsone}(P))^2$ above or by adding one to each factor in~\eqref{eq:dctr2}.

This discrepancy, for any $\bsgamma$, does not change if, for any $u\subseteq1{:}d$, we replace $x_{i,k}$ by $1-x_{i,k}$ for all $i=1,\dots,n$ and all $k\in u$.

For MC points 
$\e[ (D_2^\ctr(P_\mc))^2] =  (4^{-d}-12^{-d})/n$ by equation~\eqref{eq:dsquaredbullet}. For $n$ points at the center $\bsc$, we get $(D_2^\ctr)^2=12^{-d}$. As a result we need $n>3^d-1$ for MC points to be better than the center point. The repeated-point pathology for the centered discrepancy is even worse than the one for the extreme discrepancy.

We also get $(D_2^\ctr)^2=12^{-d}$ if there is even one coordinate $k\in1{:}d$ with $x_{i,k}=1/2$ for all $i=1,\dots,n$.  Then the points of $P$ are confined to one of $d$ hyperplanes and there is an axis-parallel open box of volume $1/2$ with no points of $P$.

\subsection{Symmetric discrepancy}\label{sec:sym}
The \emph{symmetric discrepancy} of \cite[Section 5.3]{hickdisc} is based on notions of even and odd vertices and orthants. There $E = \{\bsv\in\{0,1\}^d\mid\text{$\sum_{k=1}^dv_k$ is even}\}$ is the set of even vertices; the others are odd.
The orthant $\rect(\bsa,\bsv)$ is said to be even or odd according to $\bsv$'s parity. For any $\bsa\in[0,1]^d$ the union of its even orthants is
\[
\orth_e(\bsa) = \bigcup_{\bsv \in E} \rect(\bsa, \bsv).
\] 
The union of the other (odd) orthants is $\orth_o(\bsa)$.
The symmetric discrepancy is the RMS value of $\delta(\orth_e(\bsa))-\delta(\orth_o(\bsa))$ over $\bsa\sim U[0,1]^d$. 

\begin{definition}[Symmetric discrepancy]The $L_2$ symmetric discrepancy of a point set $P$ is
\begin{align}\label{eq:symml2disc}
(D_2^{\sym}(P))^2 = 4\int_{[0,1]^d} \delta(\orth_e(\bsa))^2 \rd\bsa.
\end{align}
\end{definition}

The factor of $4$ comes from subtracting the odd discrepancy, thereby doubling the even discrepancy, and then squaring the difference.

Section 5.3 of \cite{hickdisc} gives the following expression for the squared $L_2$ symmetric discrepancy:
\begin{multline}\label{eq:fredyssym}
\sum_{u\ne\emptyset}(D^\sym_2(P_u))^2 = \left(\frac43\right)^d-\frac2n\sum_{i=1}^n\prod_{k=1}^d\bigl( 1 + 2x_{i,k}(1-x_{i,k})\bigr)\\+\frac{2^d}{n^2}\sum_{i=1}^n\sum_{j=1}^n\prod_{k=1}^d\bigl(1-|x_{i,k}-x_{j,k}|\bigr). 
\end{multline}
A weighted version from equation~(32c) of \cite{hick1998} is
\begin{multline}\label{eq:weightedsym}
(D_2^{\sym,\bsgamma}(P))^2=
\prod_{k=1}^d \left( 1 + \frac{\gamma_k}{12} \right) 
- \frac{2}{n} \sum_{i=1}^n \prod_{k=1}^d \left( 1 + \frac{\gamma_k}{2} x_{i,k}(1 - x_{i,k}) \right) \\
+ \frac{1}{n^2} \sum_{i,j=1}^n \prod_{k=1}^d \left[ 1 + \frac{\gamma_k}{4} \bigl( 1 - 2 \left| x_{i,k} - x_{j,k} \right| \bigr) \right]
\end{multline}
so we may write
\[
(D_2^{\sym,\bsgamma}(P))^2=\sum_{u\ne\emptyset}(D_2^\sym(P_u))^2\prod_{k\in u}\frac{\gamma_k}4.
\]
Our prior sections used $\gamma_k = \hat\beta_k^2$ where $\hat\beta_k$ was from \cite{hick1998}. Section 5.3 of that paper uses $\hat\beta_k/2$ where the other sections use $\hat\beta_k$. We can define $\hat\alpha_k=\hat\beta_k/2$ and $\eta_k=\hat\alpha_k^2=\gamma_k/4$ and write the weighted symmetric discrepancy as
\begin{multline}\label{eq:weightedsym2}
(\tilde D_2^{\sym,\bseta}(P))^2= \sum_{u\ne\emptyset}(D_2^\sym(P_u))^2\prod_{k\in u}\eta_k \\=
\prod_{k=1}^d \left( 1 + \frac{\eta_k}{3} \right) 
- \frac{2}{n} \sum_{i=1}^n \prod_{k=1}^d \left( 1 + 2\eta_kx_{i,k}(1 - x_{i,k}) \right) \\
+ \frac{1}{n^2} \sum_{i,j=1}^n \prod_{k=1}^d \left[ 1 + \eta_k \bigl( 1 - 2 \left| x_{i,k} - x_{j,k} \right| \bigr) \right].
\end{multline}
Then the coefficient of $\prod_{k=1}^d\eta_k$
in~\eqref{eq:weightedsym2} gives
\begin{multline}\label{eq:sym}
(D_2^{\sym}(P))^2 = \frac{1}{3^d} - \frac2n \sum_{i=1}^n \prod_{k=1}^d 2x_{i,k}(1-x_{i,k}) + \frac{1}{n^2} \sum_{i,j=1}^n \prod_{k=1}^d (1-2 | x_{i,k} - x_{j,k}|)
\end{multline}
as the squared $L_2$ symmetric discrepancy.

For MC points $\e[(D_2^\sym(P_\mc))^2]$$=(1-3^{-d})/n$.
For $n$ identical points $(D_2^\sym)^2$  is minimized when those points equal $\bsc$ and the resulting value is $1-2^{1-d}+3^{-d}$.
For $d\ge2$, the random points are better when $n>1$. For $d=1$ they are better for $n>2$ and equally good at $n=2$.

While $D_2^\sym$ does not have the repeated-point pathology, it does have a twin-point pathology that we describe next. For $d\ge2$ and an even number $n$, choose any $n/2$ points $\bsv_i\in[0,1]^{d-1}$. Then take $P$ to include $(\bsv_i,0)$ and $(\bsv_i,1)$ for $i=1,\dots,n/2$. Then for any $\bsx_i$ in $P$, $x_{i,d}(1-x_{i,d})=0$ making the second term of~\eqref{eq:sym} vanish and the third term is
\[
\frac1{n^2}\sum_{i,j=1}^{n/2}
\prod_{k=1}^{d-1}(1-2|v_{i,k}-v_{j,k}|)
(1-1-1+1)=0
\]
so that $(D_2^\sym(P))^2=3^{-d}$. Taking $n$ random points is not better than exploiting this twin-point pathology  until $n>3^d-1$. We were alerted to the possibility of pathologies beyond repeated-point pathologies by the MPMC optimizations for $d=10$ that were able to greatly improve on Sobol' points as reported in Section~\ref{sec:dim10}. We could trigger the pathology by putting half of the points at one vertex $\bsv$ and half at some other vertex that differs from $\bsv$ in exactly one component.

The twin-point construction is a special case of two broader pathological families, both of which MPMC found for $d=8$ and $d=10$ (Section~\ref{sec:dim10}). First, suppose that every $\bsx_i$ is a vertex of $[0,1]^d$. Then the second term of~\eqref{eq:sym} vanishes and $\prod_{k=1}^d(1-2|x_{i,k}-x_{j,k}|)=s(\bsx_i)s(\bsx_j)$, where $s(\bsv)=(-1)^{\sum_{k=1}^dv_k}$. The third term is then $\bigl(n^{-1}\sum_{i=1}^ns(\bsx_i)\bigr)^2$, so $(D_2^\sym(P))^2=3^{-d}$ whenever $P$ contains equally many even and odd vertices, counting multiplicity. Second, suppose that $P$ can be split into pairs of points that differ only in one coordinate $k$, which equals $0$ for one point of the pair and $1$ for the other. The coordinate $k$ may differ from pair to pair and the other coordinates are arbitrary. Every point then has a coordinate in $\{0,1\}$, so the second term of~\eqref{eq:sym} vanishes. For any such pair and any $\bsy\in[0,1]^d$ the kernel values $\prod_{k=1}^d(1-2|x_{i,k}-y_k|)$ of its two points sum to zero, because $(1-2|0-b|)+(1-2|1-b|)=0$ for all $b\in[0,1]$. So the third term vanishes too and again $(D_2^\sym(P))^2=3^{-d}$.

\subsection{Mixture discrepancy}\label{sec:mix}

The \emph{$L_2$ mixture discrepancy} of~\cite{zhoufang2013} combines elements of the centered and periodic discrepancies. The geometric expression in terms of subsets \( A \subset [0,1]^d \) is somewhat complex and we do not present it.

The formula given in Theorem 1 of~\cite{zhoufang2013} includes the lower-dimensional projections with weights set to $1$. It is
\begin{multline*}
(D_2^{\mix,\bsone}(P))^2=\Bigl(\frac{19}{12}\Bigr)^d
-\frac2n\sum_{i=1}^n\prod_{k=1}^d\Bigl( \frac53-\frac14|x_{i,k}-1/2|-\frac14|x_{i,k}-1/2|^2\Bigr)\\
+\frac1{n^2}\sum_{i=1}^n\sum_{j=1}^n\prod_{k=1}^d\Bigl(\frac{15}8 -\frac14|x_{i,k}-1/2|-\frac14|x_{j,k}-1/2|-\frac34|x_{i,k}-x_{j,k}|+\frac12|x_{i,k}-x_{j,k}|^2\Bigr)
\end{multline*} We present here, for the first time, the general weighted formula:
\begin{multline}\label{eq:mixwt}
(D_2^{\mix, \boldsymbol{\gamma}}(P))^2 = \prod_{k=1}^d
\left(1+\frac{7 \gamma_k}{12}\right)\\
-\frac{2}{n}\sum_{i=1}^n
\prod_{k=1}^d
\left(
1+\gamma_k
\left(
\frac{2}{3}
-\frac14\left|x_{i,k}-\frac12\right|
-\frac14\left|x_{i,k}-\frac12\right|^2
\right)
\right)
\\
+\frac{1}{n^2}
\sum_{i,j=1}^n
\prod_{k=1}^d
\Biggl(
1+\gamma_k
\Biggl(
\frac{7}{8}
-\frac14\left|x_{i,k}-\frac12\right|
-\frac14\left|x_{j,k}-\frac12\right|\\
-\frac34\left|x_{i,k}-x_{j,k}\right|
+\frac12\left|x_{i,k}-x_{j,k}\right|^2
\Biggr)
\Biggr).
\end{multline}
The coefficient of $\prod_{k=1}^d\gamma_k$ in~\eqref{eq:mixwt} is the fully $d$-dimensional mixture discrepancy
\begin{multline}\label{eq:mix}
(D_2^{\mix}(P))^2 = \left( \frac{7}{12} \right)^d - \frac{2}{n} \sum_{i=1}^n \prod_{k=1}^d \left( \frac{2}{3} - \frac{1}{4} \lvert x_{i,k} - 1/2 \rvert - \frac{1}{4} ( x_{i,k} - 1/2 )^2 \right) \\
+ \frac{1}{n^2} \sum_{i,j=1}^n \prod_{k=1}^d \left( \frac{7}{8} - \frac{1}{4} \lvert x_{i,k} - 1/2 \rvert - \frac{1}{4} \lvert x_{j,k} - 1/2 \rvert - \frac{3}{4} \lvert x_{i,k} - x_{j,k} \rvert + \frac{1}{2} \lvert x_{i,k} - x_{j,k} \rvert^2 \right).
\end{multline}

By equation~\eqref{eq:dsquaredbullet}, $\e[(D_2^\mix(P_\mc))^2]=((3/4)^d-(7/12)^d)/n$.  For $n$ points at a vertex (not necessarily the minimizing location) we get
$(D_2^\mix(P))^2=(5/8)^d-2(23/48)^d+
(7/12)^d$. For MC points to be better we need $n\approx(6/5)^d$ or more. We have not seen any other pathologies for these points.

The relation between~\eqref{eq:warnock} and~\eqref{eq:l2starones} noted in Section~\ref{sec:stardiscrepancy} holds for $D_2^\ext$, $D_2^\per$, $D_2^\ctr$ and $D_2^\mix$, and also for the average squared discrepancy $D_2^\asd$ of Section~\ref{sec:averagesquared}: adding 1 to each factor of the fully $d$-dimensional formula gives the weighted formula with $\bsgamma=\bsone$. It does not hold for $D_2^\sym$ in the $\bsgamma$ parametrization of~\eqref{eq:weightedsym}. Adding 1 to each factor of~\eqref{eq:sym} gives~\eqref{eq:weightedsym2} with $\bseta=\bsone$, that is $\bsgamma=4\times\bsone$, which is $4^d(D_2^\asd)^2$ by Proposition~\ref{prop:sameasbefore}, and not the $\bsgamma=\bsone$ version used in our computations.

\subsection{Center-anchored discrepancy}

The \emph{center-anchored discrepancy} provides another $L_2$ discrepancy that
does not treat the origin differently from
other vertices. It uses the center $\bsc$ of $[0,1]^d$. It is the $L_2$ discrepancy anchored at $\bsc$ in the sense of \cite{NovWoz09}, who define the discrepancy anchored at any $\boldsymbol{\alpha}\in[0,1]^d$ through the boxes between $\bsx$ and $\boldsymbol{\alpha}$; see also Section~3.4 of \cite{PILLNOVAK2025}.
\begin{definition}[Center-anchored discrepancy]\label{def:center-anchored} The $L_2$ center-anchored discrepancy of a point set $P$ is 
\[
(D_2^{\cad}(P))^2 = \int_{[0,1]^d} \delta(\rect(\bsa, \bsc))^2 \rd\bsa.
\]
\end{definition}

We can write this as
\begin{multline}\label{eq:center2}
(D_2^{\cad}(P))^2 = \frac1{12^d} -\frac2{n}\sum_{i=1}^n\prod_{k=1}^d \frac{x_{i,k}(1-x_{i,k})}{2} \\ +\frac1{n^2}\sum_{i,j=1}^n
\ind_{\{\bsv(\bsx_i)=\bsv(\bsx_{j})\}}\prod_{k=1}^d\min\bigl( |x_{i,k}-v(x_{i,k})|,|x_{j,k}-v(x_{j,k})|\bigr).
\end{multline}
This measure has a jump discontinuity at any \( x_{i,k}= 1/2\), making it unsuitable for any gradient-based optimization. We list it here as another way to treat all vertices of $[0,1]^d$ symmetrically.  

For MC points, $\e[ (D^\cad_2(P_\mc))^2]=(4^{-d}-12^{-d})/n$ while placing all $n$ points at any vertex $\bsv$ yields $(D_2^\cad)^2=12^{-d}$. Therefore we need $n>3^d-1$ to get something better than all points at a vertex $\bsv$.  The pathology remains if for $k\in1{:}d$ all $n$ values of $x_{i,k}$ are in $\{0,1\}$.

Because $D_2^\cad$ is not suitable for computation we do not include it in our numerical examples.  We also do not develop the weighted version nor further study its pathologies.

\section{Pathologies for weighted discrepancies}\label{sec:pathoweights}
The weighted discrepancies with $\bsgamma=\gamma\times\bsone$ take the form 
\[(D_2^{\bullet,\gamma\times\bsone}(P))^2=\sum_{u\ne\emptyset}\gamma^{|u|}(D_2^{\bullet}(P_u))^2\]  
except for $D_2^\sym$ where the projected squared discrepancies are multiplied by $(\gamma/4)^{|u|}$.
When we need $n\ge \lambda^d$ for $n$ MC points to be better than $n$ pathological ones in $[0,1]^d$, we only need $n\ge \lambda^{|u|}$ in a $|u|$-dimensional coordinate projection. The weights bias the total criterion away from the most pathological case.  Unit weights greatly diminish but do not remove repeated-point pathologies.

There is a geometric interpretation. The pathology for $D_2^*$ arises when every $\bsx_i$ has a $k$ with $x_{i,k}=1$.  A projection onto a set $u\subset1{:}d$ of variables can remove such coordinates $k$. That doesn't remove the repeated-point pathology with all $\bsx_i=\bsone$ but it does reduce the set of configurations triggering that pathology. 
Unlike some other pathologies, a repeated point 
persists in every projection.

For $D_2^\ctr$ the pathology arises when some coordinate  $k\in1{:}d$ has all $x_{i,k}=1/2$.  For $D_2^\ext$ the pathology arises when some coordinate $k$ has all $x_{i,k}=0$ or has all $x_{i,k}=1$.  The twin-point pathology in $D_2^\sym$ has some coordinate $k$ with $x_{i,k}=0$ for half of $P$ and $x_{i,k}=1$ for the other half. The broader families of Section~\ref{sec:sym} need not have a single such critical coordinate.  Projections that don't include the critical coordinate $k$ destroy these pathologies while the other projections retain them.

Our discrepancies satisfy $\e[(D_2^\bullet(P_\mc))^2] = (a_1^d-a_2^d)/n$ for $a_1>a_2$ and our repeated-point point sets have $(D_2^\bullet(P))^2=b_1^d-2b_2^d+b_3^d$ for $b_1>\max(b_2,b_3)$. Then MC points are better for $n\gtrsim (a_1/b_1)^d$.  

If $P$ is a repeated-point point set then so is $P_u$ for all $u\ne\emptyset$. If the repeated point is $\bsc$ then $P_u$ repeats the center of $[0,1]^{|u|}$.  If the repeated point is a vertex then $P_u$ also repeats a vertex. In these cases the MC discrepancy is $\sum_{u\ne\emptyset}\gamma^{|u|}(a_1^{|u|}-a_2^{|u|})/n\approx(1+\gamma a_1)^d$ and the pathological discrepancy approximately $(1+\gamma b_1)^d$. Then MC points are better for $n\gtrsim ((1+\gamma a_1)/(1+\gamma b_1))^d$.  Table~\ref{tab:path} tabulates weighted and unweighted asymptotic cross-over values of $n$ for repeated-point pathologies.

\begin{table}[t]
\centering
\begin{tabular}{llll}
\toprule
$D_2^\bullet$ & Repeated point & Unweighted & $\bsgamma=\bsone$\\
\midrule
$D_2^*$ & $\bsone$& $(3/2)^d$ & $(9/8)^d$\\
$D_2^\ext$ & Any vertex & $2^d$ & $(14/13)^d$\\
$D_2^\ctr$ & $\bsc$ & $3^d$ & $(15/13)^d$\\
$D_2^\sym$ & none & $-$ & $-$\\
$D_2^\mix$ & Any vertex & $(6/5)^d$ & $(14/13)^d$\\
$D_2^\per$ & none & $-$ & $-$\\
$D_2^\cad$ & Any vertex  & $3^d$ & Not studied\\
\bottomrule
\end{tabular}
\caption{\label{tab:path}
Summary of repeated-point pathologies that remain repeated-point pathologies after projection. $D_2^\sym$ and $D_2^\per$ have no repeated-point pathologies while $D^*_2$, $D^\ext_2$ and $D_2^\ctr$ have additional pathologies, with all points in a small number of hyperplanes, that projection can remove. The last two columns give asymptotic cross-over sample sizes $n$ for unweighted and equally weighted cases. 
}
\end{table}

We don't have explicit formulas for how weighting attenuates the more general pathologies, such as $D_2^\ctr$ being pathological if all $x_{i,k}=1/2$ for a critical $k$. 
For example, suppose that $P = Q\times\{1/2\}$, for a point set $Q$, so $\bsx_i = (\bsv_i,1/2)$ for $n$ points $\bsv_i\in[0,1]^{d-1}$. Then $D_2^{\ctr,\gamma\times\bsone}(P)$ depends on $D_2^{\ctr,\gamma\times\bsone}(Q)$. The only very small $D_2^{\ctr,\gamma\times\bsone}(Q)$ we know of has $v_{i,\ell}=1/2$ for all $i=1,\dots,n$ and some $\ell\in\{1,2,\dots,d-1\}$. Applying the decomposition recursively leads back to the repeated-point case without giving a new $\lambda$.

\section{The average squared discrepancy}\label{sec:averagesquared}

Another way of eliminating the special status of the origin is by treating
each vertex of $[0,1]^d$ in turn as the origin and averaging. As mentioned previously, this was first done analogously for the $L_{\infty}$ star discrepancy in~\cite{CDKP}, where it is called the \emph{multiple-corner} discrepancy.
For $u \subseteq 1{:}d$ and $i=1,\dots,n$, let
\[
x^u_{i,k} = \begin{cases} x_{i,k}, & k\in u\\
  1-x_{i,k}, & \text{else}
\end{cases}
\]
be a partial reflection of the point $\bsx_i$.
\begin{definition}[Average squared discrepancy]
The $L_2$ \emph{average squared discrepancy} of a point set $P$ is
\begin{align}\label{eq:symdstar2}
(D^{\asd}_2(P))^2 
= \frac1{2^d}\sum_{u\subseteq1:d}D_2^*(\bsx_1^u,\dots,\bsx_n^u)^2.
\end{align}
\end{definition}
For MC points the distribution of $(\bsx_1^u,\dots,\bsx_n^u)$ is the same for all $u\subseteq1{:}d$. Then
$\e[ (D^\asd_2(P_\mc))^2]=\e[ (D^*_2(P_\mc))^2] =(2^{-d}-3^{-d})/n$. 

\subsection{Cost of ASD}
Although \eqref{eq:symdstar2} involves a summation over $2^d$ terms, it can be computed with the same computational complexity as the other $L_2$ discrepancies considered thus far. The next theorem gives a Warnock-style formula for it.
\begin{theorem}\label{thm:itsodnsq}
In the above notation
\begin{align}\label{eq:d2forasd}
  (D^{\asd}_2(P))^2 &= \frac1{3^d}-
  \frac2n\sum_{i=1}^n\prod_{k=1}^d\frac{1+2x_{i,k}-2x_{i,k}^2}4 
+\frac1{n^2}\sum_{i,j=1}^n\prod_{k=1}^d 
\frac{1-|x_{i,k}-x_{j,k}|}2.
\end{align}
\end{theorem}

\begin{proof}
Plugging the Warnock formula~\eqref{eq:warnock} into
the definition~\eqref{eq:symdstar2} shows that $(D_2^\asd(P))^2$ equals
\begin{align*}
  \frac1{3^d}-
  \frac2n\sum_{i=1}^n\sum_{u\subseteq1:d}\prod_{k=1}^d\frac{1-(x^u_{i,k})^2}4
+\frac1{n^2}\sum_{i,j=1}^n\sum_{u\subseteq1:d}\prod_{k=1}^d
\frac{\min\bigl( 1-x^u_{i,k},1-x^u_{j,k}\bigr)}2.
\end{align*}
Now
\begin{align*}
\sum_{u\subseteq1:d}\prod_{k=1}^d\frac{1-(x^u_{i,k})^2}4
&=\prod_{k=1}^d\frac{1-x_{i,k}^2+1-(1-x_{i,k})^2}4
=\prod_{k=1}^d\frac{1+2x_{i,k}-2x_{i,k}^2}4
\end{align*}
and then
\begin{align*}
&\phantom{=}\,\,\sum_{u\subseteq1:d}\prod_{k=1}^d
\frac{\min\bigl( 1-x^u_{i,k},1-x^u_{j,k}\bigr)}2\\
&=
\prod_{k=1}^d\frac{
\min\bigl( 1-x_{i,k},1-x_{j,k}\bigr)
+\min\bigl( 1-(1-x_{i,k}),1-(1-x_{j,k})\bigr)
}2\\
&=\prod_{k=1}^d\frac{1+\min(x_{i,k},x_{j,k})-\max(x_{i,k},x_{j,k})}2
\end{align*}
from which the result follows.
\end{proof}

\subsection{ASD as a symmetric discrepancy}
We observe the following unexpected equivalence between the ASD and the weighted symmetric discrepancy from \cite{hick1998}, despite the differences in how these measures are constructed.

\begin{proposition}\label{prop:sameasbefore}
In the above notation,
\[
4^d (D_2^{\asd}(P))^2 = (D_2^{\sym, \bsgamma}(P))^2,
\]
where \( \bsgamma  = (4, \ldots, 4) \).
\end{proposition}

\begin{proof}
This identity follows by direct comparison of equations \eqref{eq:weightedsym} and \eqref{eq:d2forasd}.
\end{proof}

For consistency with the rest of the manuscript, we also provide a Warnock-style formula for the weighted variant of the average squared discrepancy. First we define
\[
(D_2^{\asd,\bsgamma}(P))^2 = \sum_{u\ne\emptyset}(D_2^\asd(P_u))^2\prod_{k\in u}\gamma_k.
\]
Then the weighted formula for it is
\begin{multline}\label{eq:weightedasd}
(D^{\asd,\bsgamma}_2(P))^2 = \prod_{k=1}^d \left( 1 + \frac{\gamma_k}{3} \right) - \frac{2}{n} \sum_{i=1}^n \prod_{k=1}^d \left( 1 + \frac{\gamma_k}{4}(1 + 2x_{i,k} - 2x_{i,k}^2) \right) \\ 
+ \frac{1}{n^2} \sum_{i=1}^n \sum_{j = 1}^n \prod_{k=1}^d \left(1 + \frac{\gamma_k}{2} \left(1 - |x_{i,k} - x_{j,k}|\right)\right),
\end{multline}
after comparing the coefficient of $\prod_{k\in u}\gamma_k$ above to equation~\eqref{eq:d2forasd}.

The weighted criteria have the same interpretation as averages over reflections. Averaging~\eqref{eq:wtdl2star} over the $2^d$ partial reflections as in~\eqref{eq:symdstar2} gives~\eqref{eq:weightedasd} with the same weights:
\[
(D_2^{\asd,\bsgamma}(P))^2=\frac1{2^d}\sum_{u\subseteq1:d}\bigl(D_2^{*,\bsgamma}(\bsx_1^u,\dots,\bsx_n^u)\bigr)^2.
\]
Comparing~\eqref{eq:weightedasd} with~\eqref{eq:weightedsym} shows that, for $0<\gamma<4$,
\[
(D_2^{\sym,\gamma\times\bsone}(P))^2=\Bigl(1-\frac{\gamma}{4}\Bigr)^d\frac1{2^d}\sum_{u\subseteq1:d}\bigl(D_2^{*,\gamma'\times\bsone}(\bsx_1^u,\dots,\bsx_n^u)\bigr)^2,\quad \gamma'=\frac{4\gamma}{4-\gamma}.
\]
So every equally weighted symmetric discrepancy with $\gamma<4$ is, up to a constant factor, an average of $2^d$ reflected weighted star discrepancies. The case $\bsgamma=\bsone$ used in our computations corresponds to $\gamma'=4/3$, and the unweighted ASD is the limit $\gamma=4$. Similarly $(D_2^{\asd,\bsone}(P))^2=(5/4)^d(D_2^{\sym,(4/5)\times\bsone}(P))^2$.

\subsection{ASD and pathology}
Here we show that the ASD does not suffer from Matoušek's repeated-point pathology.  It is prone to a recursive twin-point pathology that is milder than the twin-point pathology of $D_2^\sym$.

\begin{proposition}\label{prop:asd_is_good}
Let $P$ contain $n$ copies of any point $\bsz\in[0,1]^d$. Then
\[
\e[ D_2^\asd(P_\mc)^2] < D_2^\asd(P)^2
\]
for \( d\ge2\) and all \( n > 1 \). For $d=1$ we need $n>2$.
\end{proposition}
\begin{proof}
As we saw in Section~\ref{sec:sym}, this holds for $D_2^\sym$, so it holds for any projection $P_u$ and any positively weighted combination thereof. When $d\ge2$ the weighted combinations include at least one component with a strict inequality for $n>1$ making the inequality strict for the sum over projections. Therefore the result holds for $(D_2^\asd)^2=4^{-d}(D_2^{\sym,\bsgamma})^2$ with $\bsgamma=4\times\bsone$. 
\end{proof}

The recursive twin-point pathology of ASD works as follows. Let $n$ be even and let $P_1$ be a point set of $n/2$ points in $[0,1]^{d-1}$ for $d\ge2$. Now let $P=P_0$ have points $(\bsv_i,0)$ and $(\bsv_i,1)$ for $\bsv_i\in P_1$. Then from equation~\eqref{eq:d2forasd}
\begin{align*}
(D_2^\asd(P))^2 &=
\frac1{3^d}-\frac{2}{n}\sum_{i=1}^{n/2}\left(\prod_{k=1}^{d-1}\frac{1+2v_{i,k}-2v_{i,k}^2}4\right)(1/4+1/4)\\
&+\frac1{n^2}\sum_{i,j=1}^{n/2}
\left(\prod_{k=1}^{d-1}\frac{1-|v_{i,k}-v_{j,k}|}2\right)(1/2+0+0+1/2)
\end{align*}
from which
\begin{align}\label{eq:twinasd}
(D_2^\asd(P_0))^2 = \frac{(D_2^\asd(P_1))^2}4 + \frac14\frac1{3^d}
\end{align}
after noting that the leading term for $(D_2^\asd(P_1))^2$ is $1/3^{d-1}$.

For the twinned point set's $D_2^\asd(P_0)$ to be unusually small requires $D_2^\asd(P_1)$ to be unusually small.  If $P_1$ was formed by twinning a smaller set $P_2$ then
\[
(D_2^\asd(P_0))^2 = \frac{(D_2^\asd(P_2))^2}{16} + \frac1{16}\frac1{3^{d-1}}+\frac1{4}\frac1{3^d}.
\]
For $r\le d-1$ we can twin the points $r$ times getting
\[
(D_2^\asd(P_0))^2 = \frac{(D_2^\asd(P_r))^2}{4^r} + 
\frac1{3^d}\Bigl(1-\Bigl(\frac34\Bigr)^r\Bigr).
\]

If we start with $P_r$ having just one point at $\bsc\in[0,1]^{d-r}$ we get a point set $P_{d,r} = \{0,1\}^r\times \{1/2\}^{d-r}$ with $n=2^r$ points in $[0,1]^d$. Then
\begin{align*}
(D_2^\asd(P_{d,r}))^2 &= 
\frac{2^{-(d-r)}-2(3/8)^{d-r}+3^{-(d-r)}}{4^r}+\frac1{3^d}\Bigl(1-\Bigl(\frac34\Bigr)^r\Bigr)\\
&=\frac1{3^d}-2\times4^{-r}\Bigl(\frac38\Bigr)^{d-r}+2^{-(d+r)}
\end{align*}
while $\e[(D_2^\asd(P_{\mc}))^2]$$=(2^{-d}-3^{-d})/n=(2^{-d}-3^{-d})/2^r$. Because
\[
(D_2^\asd(P_{d,r}))^2-
\e\left[(D_2^\asd(P_{\mc}))^2\right]
=
3^{-d}(1+2^{-r})
-2\cdot4^{-r}\left(\frac38\right)^{d-r}
\]
we get $\e[(D_2^\asd(P_\mc))^2]<(D_2^\asd(P_{d,r}))^2$ if and only if
\[
2\Bigl(\frac23\Bigr)^r\Bigl(\frac98\Bigr)^d<1+2^{-r}.
\]
That requires $(3/2)^r>(9/8)^d$, or $r>d\alpha$ for $\alpha=\log(9/8)/\log(3/2)\approx 0.29$. So $n$ must be at least $(2^\alpha)^d\sim 1.223^d$ for $P_\mc$ to be better than $P_{d,r}$. 

In Section~\ref{sec:pathoweights} we were unable to use a recursive construction to see whether weighting made very general pathologies milder.  The setting for ASD is different because we are able to combine recursive twinning with an initial center point construction. While the pathology for $D_2^\sym$  persists until $3^d$, the pathology we find for ASD or equivalently $D_2^{\sym,4\times\bsone}$, persists until $1.223^d$.

The pathological points in this example resemble the Box-Behnken \cite{box:behn:1960} design for fitting response surfaces.  Some of those designs are unions of doubly twinned center points (i.e., with $r=2$).

\section{Numerical Results}\label{sec:numerical}

We present a comprehensive numerical study employing recently introduced state-of-the-art sample point optimization techniques using a range of objective functions introduced above, both for weighted and unweighted discrepancies. Our analysis focuses on comparing resulting point sets in a range of dimensions between 2 and 10 for $n$ as powers of $2$ using the Sobol' sequence as a benchmark. As our optimization framework, we use the MPMC method \cite{rusc:etal:2024}, a deep learning global optimization method, and a recent greedy approach from Kritzinger \cite{krit:2022}, with extended numerical results in~\cite{clem:2023}, as well as a gradient descent method~\cite{FCL2Grad}. Our comparison reference is the Sobol' sequence from the QMCPy Python package \cite{qmcpy}.

\subsection{Message-Passing Monte Carlo}\label{sec:mpmc}
MPMC represents the first use of deep learning methods, specifically leveraging a message-passing graph neural network model, to generate point sets contained in the $d$-dimensional unit hypercube with very small discrepancy. Importantly, in its original formulation, the training objective for MPMC is 
$(D_2^*)^2$ 
given by equation~\eqref{eq:warnock}. MPMC has later been extended to general distributions via a kernelized Stein discrepancy approach \cite{kirk2025lowsteindiscrepancymessagepassing}. 

All MPMC experiments were conducted on NVIDIA DGX A100 and NVIDIA H100 GPUs. Models were trained using Adam~\cite{Kingma2014AdamAM} with both weight decay and learning rate decay for a minimum of 100k steps; training was stopped once the learning rate reached a value less than $10^{-6}$. MPMC hyperparameters were tuned via random search over 15 sampled configurations. Specifically, the learning rate was drawn from a log-uniform distribution over $[5\times10^{-6}, 10^{-3}]$; the weight decay factor was drawn from a log-uniform distribution over $[10^{-10}, 10^{-5}]$; the number of GNN layers was sampled uniformly from $\{3,4,5,6\}$; the hidden dimension from $\{64,128,256\}$; and the graph radius from a uniform distribution over $[0.1, 0.7]$; the batch size was drawn from $\{2,4,8\}$. The MPMC algorithm is initialized with a random set of uniformly distributed points.
For further model details, we refer the reader to \cite{rusc:etal:2024}.

We extend the MPMC framework to directly minimize discrepancies $(D_2^\bullet)^2$ for $\bullet \in \{*, \ext, \sym, \ctr, \per, \asd, \mix\}$ and their weighted variants. For any weighted criterion, we simply set $\boldsymbol{\gamma}$ to the all-one vector.  Even unit weights imply a kind of low effective dimension \cite{presobolev} when used to define an RKHS.

All MPMC results tabulated in this section and in the Appendix are based on discrepancies $D_2^\bullet$, not squared discrepancies, for the single best point set found for each criterion, dimension and sample size. There is no averaging over replicates. The values were evaluated in double precision. Discrepancies are truncated to three significant figures and ratios to two decimal places, except for the very small centered ratios at $d=10$, which are given to three significant figures. The unweighted criteria use equations~\eqref{eq:warnock}, \eqref{eq:ext}, \eqref{eq:per}, \eqref{eq:dctr2}, \eqref{eq:sym}, \eqref{eq:d2forasd} and~\eqref{eq:mix}. For the symmetric discrepancy this is $D_2^\sym$ as defined in~\eqref{eq:sym}; in the normalization of~\eqref{eq:weightedsym} it would be $2^{-d}D_2^\sym$. The weighted criteria use equations~\eqref{eq:l2starones}, \eqref{eq:extwt}, \eqref{eq:perwt}, \eqref{eq:weightedctr}, \eqref{eq:weightedsym}, \eqref{eq:weightedasd} and~\eqref{eq:mixwt}, all with $\bsgamma=\bsone$.

We train $d$-dimensional point sets for $n = 2^m$ with $m = 4$ to $m=10$ and $d\in \{2,3,4,5,8,10\}$. While complete result tables are included in the Appendix, the following subsections describe the main observations of our results in dimensions 2, 5 and 10. Section~\ref{sec:graddesc} describes some complementary results using a gradient descent approach, and Section~\ref{sec:greedy} describes greedy sequential constructions. Code for the MPMC model can be found at \url{https://github.com/tk-rusch/MPMC}.

\subsubsection{Dimension 2}
We begin with a ratio comparison to the Sobol' sequence in dimension 2 for unweighted discrepancies, with the results presented in Table~\ref{tab:res:2d_comp}. MPMC reliably reduces the discrepancies of all criteria and values of $n$ that we consider with at least a 20\% improvement for all measures other than periodic and mixture. The improvements are smaller, but still significant, for larger instances of $n$. The greatest improvements on the Sobol' values come for low $n$ instances. The stronger performance of MPMC for low $n$ was already observed in the original paper~\cite{rusc:etal:2024}. This is a consistent trend across all the results presented in this paper.

\begin{table}[t]
\centering
\begin{tabular}{rccccccc}
\toprule
$n$ & $D_2^{*}$& $D_2^{\ext}$& $D_2^{\per}$& $D_2^{\ctr}$& $D_2^{\sym}$&$D_2^{\asd}$&$D_2^{\mix}$ \\
\midrule
16& $0.52$ & $0.63$ & $0.70$ & $0.66$ & $0.54$ & $0.64$ & $0.72$\\
32& $0.55$ & $0.78$ & $0.87$ & $0.74$ & $0.62$ & $0.69$ & $0.81$\\
64& $0.56$ & $0.76$ & $0.84$ & $0.72$ & $0.64$ & $0.70$ & $0.81$\\
128& $0.55$ & $0.76$ & $0.84$ & $0.71$ & $0.65$ & $0.69$ & $0.83$\\
256& $0.63$ & $0.68$ & $0.74$ & $0.72$ & $0.66$ & $0.71$ & $0.78$\\
512& $0.63$ & $0.71$ & $0.75$ & $0.72$ & $0.66$ & $0.73$ & $0.83$\\
1024& $0.65$ & $0.76$ & $0.92$ & $0.77$ & $0.68$ & $0.79$ & $0.88$\\
\bottomrule
\end{tabular}
 \caption{Ratios of the unweighted discrepancy values obtained via MPMC compared to those of the Sobol' sequence in dimension 2 for varying $n$. Values below 1 indicate that MPMC outperforms the Sobol' sequence. }
\label{tab:res:2d_comp}
\end{table}

The results for weighted discrepancies presented in Table~\ref{tab:mpmc:2d_compwei} show similar tendencies. MPMC reduces the discrepancy in most instances. The weighted extreme discrepancy is the exception: for $n\ge64$ MPMC is within about 5\% of the Sobol' points, and at $n=1024$ it is 6\% worse. The weighted periodic discrepancy at $n=1024$ is also 3\% worse. In this regime $(D_2^{\bullet,\bsgamma})^2$ is of order $10^{-7}$ while the constant term of its Warnock-type formula is of order one, so it can only be resolved in double precision. MPMC's training objective is computed in single precision, where these differences are dominated by rounding error. 

\begin{table}[t!]
\centering
\begin{tabular}{rccccccc}
\toprule
$n$ & $D_2^{*,\boldsymbol{\gamma}}$& $D_2^{\ext,\boldsymbol{\gamma}}$& $D_2^{\per,\boldsymbol{\gamma}}$& $D_2^{\ctr,\boldsymbol{\gamma}}$& $D_2^{\sym,\boldsymbol{\gamma}}$& $D_2^{\asd,\boldsymbol{\gamma}}$ &$D_2^{\mix,\boldsymbol{\gamma}}$ \\
\midrule
16& $0.54$ & $0.88$ & $0.84$ & $0.59$ & $0.58$ & $0.57$ & $0.73$\\
32& $0.58$ & $0.98$ & $0.97$ & $0.65$ & $0.63$ & $0.62$ & $0.80$\\
64& $0.59$ & $1.00$ & $0.98$ & $0.65$ & $0.65$ & $0.63$ & $0.82$\\
128& $0.59$ & $1.00$ & $0.96$ & $0.66$ & $0.66$ & $0.66$ & $0.83$\\
256& $0.64$ & $0.95$ & $0.88$ & $0.68$ & $0.69$ & $0.66$ & $0.83$\\
512& $0.65$ & $1.00$ & $0.90$ & $0.68$ & $0.70$ & $0.66$ & $0.81$\\
1024& $0.74$ & $1.05$ & $1.02$ & $0.76$ & $0.83$ & $0.77$ & $0.87$\\
\bottomrule
\end{tabular}
 \caption{Ratios of the weighted discrepancy values obtained via MPMC compared to those of the Sobol' sequence in dimension 2 for varying $n$. Values below 1 indicate that MPMC outperforms the Sobol' sequence. All weighted criteria use $\bsgamma=\bsone$. }
\label{tab:mpmc:2d_compwei}
\end{table}

Given the numerous measures that exist for quantifying uniform distribution, an important element to consider is if optimizing for one measure allows for the optimization of multiple measures at once. Table~\ref{tab:nonweighted_d2_n256} describes, for each point set optimized for a measure $D_2^{\bullet}$, the discrepancy values for \emph{all} considered measures. From this, we can see that the ASD and the mixture discrepancy are nearly always the two best performing discrepancies for each category (apart from the set optimized for that measure). In particular, the ASD always does very well for the measures, with never more than a 20\% loss compared to the set optimized for the measure (and within 2\% of the mixture-optimized set on $D_2^\mix$).

\begin{table}[t]
\centering
\resizebox{\columnwidth}{!}{
\begin{tabular}{cccccccc}
\toprule
Set& $D_2^{*}$& $D_2^{\ext}$& $D_2^{\per}$& $D_2^{\ctr}$& $D_2^{\sym}$ & $D_2^{\asd}$ & $D_2^{\mix}$ \\
\midrule
\textasteriskcentered & $0.00210$ &$0.00172$ &$0.00437$ &$0.00503$ &$0.01456$ &$0.00507$ &$0.00518$ \\
ext& $0.00704$ &$0.00146$ &$0.00520$ &$0.00666$ &$0.02506$ &$0.00664$ &$0.00545$ \\
per& $0.00399$ &\textbf{0.00155} &$0.00326$ &$0.00452$ &$0.01880$ &$0.00485$ &$0.00417$ \\
ctr &$0.00389$ &$0.00306$ &$0.00566$ &$0.00184$ &$0.01274$ &$0.00397$ &$0.00619$ \\
sym &$0.00678$ &$0.00157$ &$0.01232$ &$0.00563$ &$0.00672$ &$0.00673$ &$0.01143$ \\
asd &\textbf{0.00215} &$0.00172$ &$0.00375$ &\textbf{0.00206} &\textbf{0.00780} &$0.00215$ &\textbf{0.00311} \\
mix &$0.00223$ &$0.00168$ &\textbf{0.00359} &$0.00213$ &$0.00827$ &\textbf{0.00223} &$0.00307$ \\
Sobol'&0.00331& 0.00214& 0.00437& 0.00254& $0.01016$& 0.00300& 0.00390\\
\bottomrule
\end{tabular}
}
\caption{Cross-evaluation of non-weighted point sets for $d=2$ and $n=256$. For each set optimized for $\{*,\ext,\per,\sym,\ctr,\asd,\mix\}$, we evaluate it on the other discrepancies. Each row corresponds to the criterion optimized while creating the point set, while each column is the criterion that the sets are evaluated for. In bold is the best set not optimized for the specific measure. }
\label{tab:nonweighted_d2_n256}
\end{table}

One may point out that, given Proposition~\ref{prop:sameasbefore}, the ASD may behave as more of a weighted measure, and that could be what makes it a good choice. Table~\ref{tab:2dcomp_wei} shows that it is reliably the best also among weighted measures. The only case where it is beaten is for the periodic discrepancy where using the \emph{weighted} periodic discrepancy is better.

\begin{table}[t]
\centering
\resizebox{\columnwidth}{!}{
\begin{tabular}{cccccccc}
\toprule
Set & $D_2^{*}$& $D_2^{\ext}$& $D_2^{\per}$& $D_2^{\ctr}$& $D_2^{\sym}$ & $D_2^{\asd}$ &$D_2^{\mix}$ \\
\midrule
\textasteriskcentered,$\boldsymbol{\gamma}$& $0.00239$ &$0.00195$ &$0.00407$ &$0.00236$ &$0.00931$ &$0.00247$ &$0.00342$ \\
ext,$\boldsymbol{\gamma}$& $0.00863$ &$0.00185$ &$0.00386$ &$0.00786$ &$0.03147$ &$0.00793$ &$0.00550$ \\
per,$\boldsymbol{\gamma}$& $0.01106$ &$0.00179$ &$0.00370$ &$0.01107$ &$0.04428$ &$0.01110$ &$0.00697$ \\
ctr,$\boldsymbol{\gamma}$& $0.00240$ &$0.00192$ &$0.00395$ &$0.00223$ &$0.00904$ &\textbf{0.00240} &$0.00336$ \\
sym,$\boldsymbol{\gamma}$& $0.00243$ &$0.00200$ &$0.00418$ &$0.00232$ &$0.00912$ &$0.00242$ &$0.00344$ \\
asd & \textbf{0.00215} &\textbf{0.00172} &\textbf{0.00375} &\textbf{0.00206} &\textbf{0.00780} &$0.00215$ & \textbf{0.00311} \\
mix,$\boldsymbol{\gamma}$& $0.00246$ &$0.00191$ &$0.00397$ &$0.00234$ &$0.00924$ &$0.00245$ &$0.00337$ \\
\bottomrule
\end{tabular}
}
\caption{Cross-evaluation of \textbf{weighted} optimized point sets (except for the ASD, which already corresponds to a weighted symmetric discrepancy) for $d=2$ and $n=256$. For each set optimized for $\{(*,\boldsymbol{\gamma}),(\ext,\boldsymbol{\gamma}),(\per,\boldsymbol{\gamma}),(\sym,\boldsymbol{\gamma}), (\ctr,\boldsymbol{\gamma}),\asd, (\mix,\boldsymbol{\gamma})\}$, we evaluate it on the \textbf{unweighted discrepancies}. In bold is the best set in each column, excluding the set optimized for that criterion (weighted or not). The ASD set is best in every other column; the weighted periodic set has the smallest $D_2^\per$ overall.  }
\label{tab:2dcomp_wei}
\end{table}

We turn now to the star discrepancy, which is the most common optimization target. Figure~\ref{fig:stardisc} compares the star discrepancy ratio between that of the set obtained by optimizing a measure and the one of the set directly optimized for the $L_2$ star discrepancy.  In both panels the ASD set has the smallest $D_2^*$ among all sets not optimized for $D_2^*$, within 2 to 9\% of the $D_2^*$-optimized set for every $n$. It is closer than the set optimized for the weighted star discrepancy (10 to 32\%) and the one optimized for the weighted symmetric discrepancy, which as Proposition~\ref{prop:sameasbefore} shows differs only in the weight choice. Weighting helps the symmetric and centered criteria considerably, roughly halving their excess $D_2^*$, but not the extreme, periodic or mixture criteria, whose weighted versions do worse than their unweighted ones for most $n$. While we do not give similar plots for the next settings, the ASD set also had the smallest $D_2^*$ among all sets not optimized for $D_2^*$ for every $n$ in dimensions 3, 4 and 5, within 2 to 16\% of the $D_2^*$-optimized set. In dimensions 8 and 10 no single criterion was consistently second best. Another change with dimension is that, as $d$ increases, the weighted symmetric and ASD sets start performing very similarly. Recall that these are $D_2^{\sym,\bsone}$ and $2^{-d}D_2^{\sym,4\times\bsone}$ respectively.

\begin{figure}[t]
    \centering
    \includegraphics[width=0.49\linewidth]{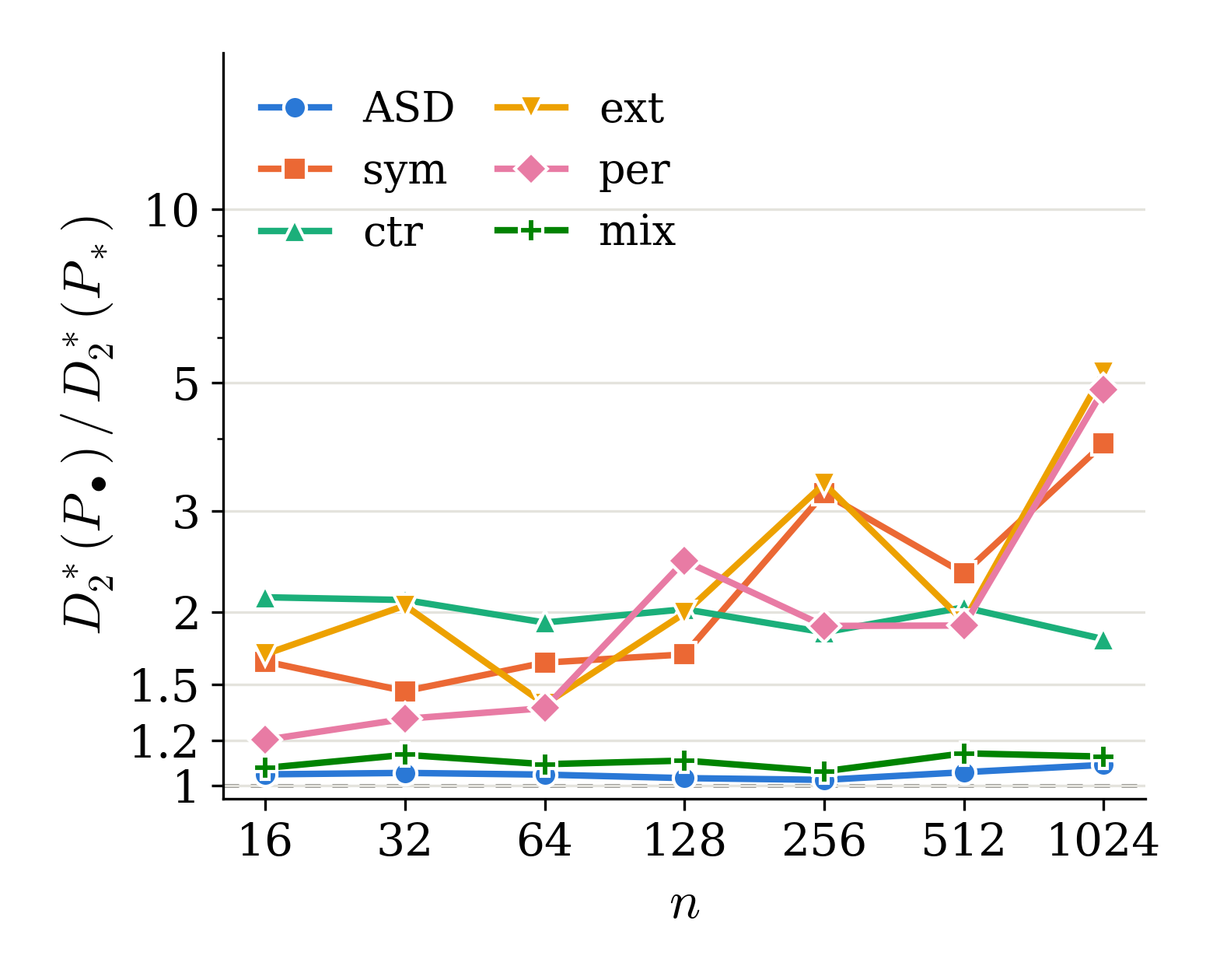}
    \hfill
    \includegraphics[width=0.49\linewidth]{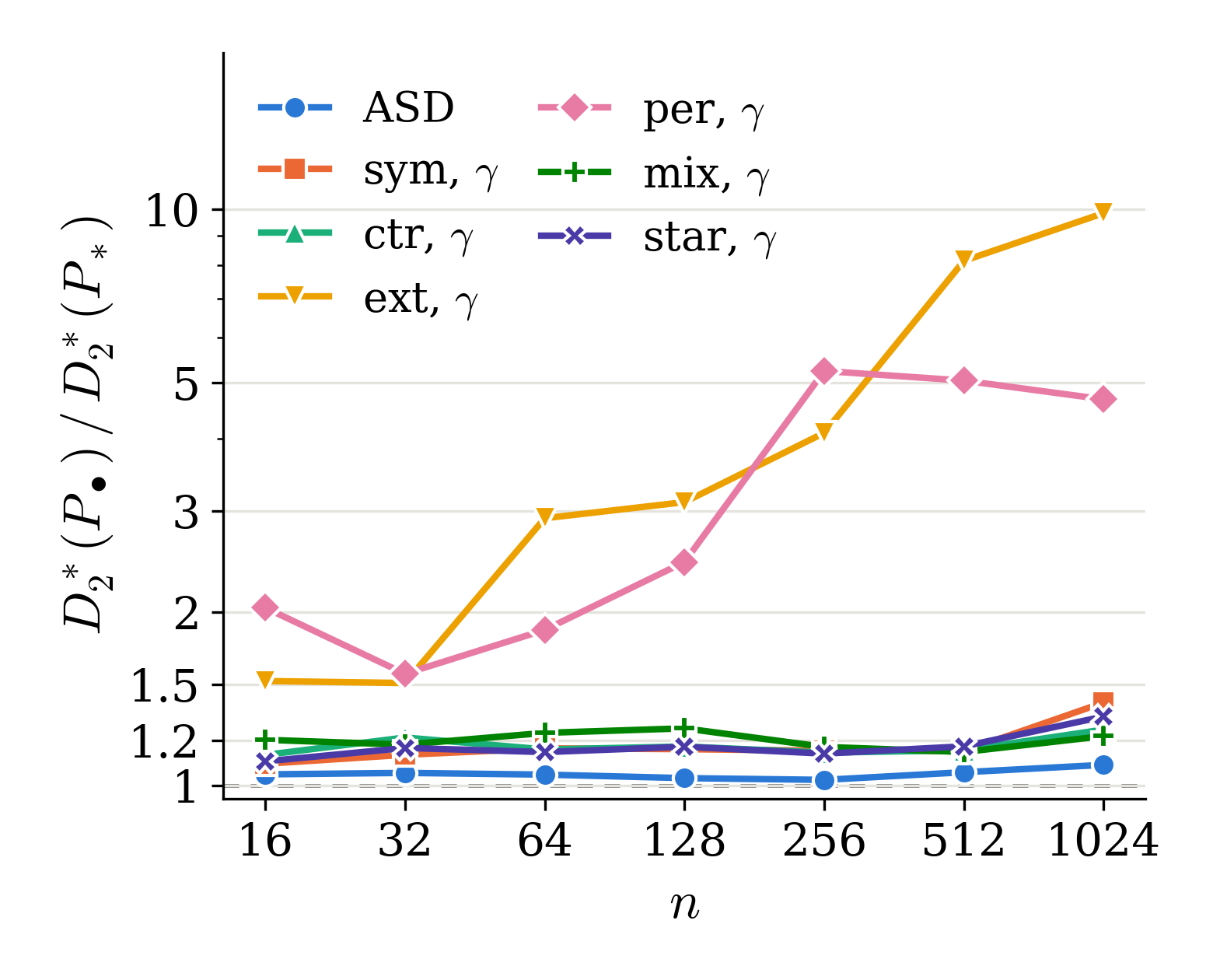}
    \caption{$D_2^*(P_{\bullet})/D^*_2(P_{*})$ for point sets $P_{\bullet}$ optimized for the different measures considered, for $d=2$. Each point is a ratio of unsquared discrepancies of the single best MPMC point sets at that $n$. Left: sets optimized for the unweighted criteria. Right: sets optimized for the weighted criteria with $\bsgamma=\bsone$, together with the ASD set. Both panels use the same logarithmic vertical scale. }
    \label{fig:stardisc}
\end{figure}

\subsubsection{Dimension 5}

We repeat the process above in five dimensions where the observations from Table \ref{tab:res:5d_comp} are very much similar; on the whole, the lower $n$ is, the better our approach performs relative to the Sobol' baseline. The first difference is that the percentage improvement is much larger, clear even at $n=1024$, however there is a faster decay in performance as $n$ increases. The second is that the difference between the measures is starker. While we improve on the periodic discrepancy by only $8$ to $15$\%, the star, symmetric and centered regularly improve by a factor of 2 or more on the Sobol' sequence.

\begin{table}[t]
\centering
\begin{tabular}{rccccccc}
\toprule
$n$ & $D_2^{*}$& $D_2^{\ext}$& $D_2^{\per}$& $D_2^{\ctr}$& $D_2^{\sym}$&$D_2^{\asd}$&$D_2^{\mix}$ \\
\midrule
16& $0.36$ & $0.68$ & $0.91$ & $0.16$ & $0.26$ & $0.81$ & $0.86$\\
32& $0.46$ & $0.83$ & $0.92$ & $0.29$ & $0.32$ & $0.77$ & $0.86$\\
64& $0.53$ & $0.85$ & $0.90$ & $0.44$ & $0.35$ & $0.76$ & $0.85$\\
128& $0.62$ & $0.83$ & $0.89$ & $0.58$ & $0.38$ & $0.81$ & $0.89$\\
256& $0.63$ & $0.83$ & $0.91$ & $0.65$ & $0.39$ & $0.78$ & $0.89$\\
512& $0.68$ & $0.79$ & $0.89$ & $0.66$ & $0.39$ & $0.81$ & $0.88$\\
1024& $0.66$ & $0.76$ & $0.85$ & $0.65$ & $0.44$ & $0.77$ & $0.88$\\
\bottomrule
\end{tabular}
 \caption{Same approach as in Table~\ref{tab:res:2d_comp}, this time in dimension 5.}
\label{tab:res:5d_comp}
\end{table}

In the weighted criterion setting, while the percentage improvements are smaller than in the two-dimensional case, many of the previous conclusions hold. That is, the weighted star, symmetric and centered discrepancies seem to be most amenable to MPMC optimization showing the biggest improvements while weighted extreme and periodic observe minimal reductions in discrepancy over Sobol'. Once again, percentage improvement deteriorates with increasing $n$ on the whole.

\begin{table}[t]
\centering
\begin{tabular}{rccccccc}
\toprule
$n$ & $D_2^{*,\boldsymbol{\gamma}}$& $D_2^{\ext,\boldsymbol{\gamma}}$& $D_2^{\per,\boldsymbol{\gamma}}$& $D_2^{\ctr,\boldsymbol{\gamma}}$& $D_2^{\sym,\boldsymbol{\gamma}}$& $D_2^{\asd,\boldsymbol{\gamma}}$ &$D_2^{\mix,\boldsymbol{\gamma}}$ \\
\midrule
16& $0.58$ & $0.91$ & $0.93$ & $0.66$ & $0.76$ & $0.76$ & $0.85$\\
32& $0.65$ & $0.93$ & $0.92$ & $0.74$ & $0.77$ & $0.77$ & $0.87$\\
64& $0.67$ & $0.96$ & $0.91$ & $0.75$ & $0.77$ & $0.76$ & $0.88$\\
128& $0.74$ & $0.94$ & $0.92$ & $0.86$ & $0.84$ & $0.85$ & $0.90$\\
256& $0.74$ & $0.92$ & $0.91$ & $0.82$ & $0.83$ & $0.82$ & $0.91$\\
512& $0.78$ & $0.92$ & $0.92$ & $0.86$ & $0.85$ & $0.84$ & $0.90$\\
1024& $0.80$ & $0.89$ & $0.89$ & $0.84$ & $0.83$ & $0.83$ & $0.88$\\
\bottomrule
\end{tabular}
 \caption{Same approach as in Table~\ref{tab:mpmc:2d_compwei}, this time in dimension 5, with $\bsgamma=\bsone$.}
\label{tab:mpmc:5d_compwei}
\end{table}

As in the two-dimensional case, we also provide cross-evaluations of the different unweighted discrepancy measures. The ASD and the mixture discrepancies once again stand out. However, we do notice larger gaps between the quality of the set optimized for a specific measure and the second best one: for symmetric, the average squared discrepancy set is twice as bad, while there is also a 32\% gap for the centered case. Table~\ref{tab:weighted_d5_n256} gives the corresponding cross-evaluation for the weighted point sets. As in dimension 2, the ASD set is best among the sets not optimized for the criterion in the $D_2^*$, $D_2^\ext$, $D_2^\ctr$ and $D_2^\sym$ columns, and the weighted periodic set has the smallest $D_2^\per$ overall.

\begin{table}[t]
\centering
\resizebox{\columnwidth}{!}{
\begin{tabular}{cccccccc}
\toprule

Set & $D_2^{*}$& $D_2^{\ext}$& $D_2^{\per}$& $D_2^{\ctr}$& $D_2^{\sym}$ & $D_2^{\asd}$ & $D_2^{\mix}$ \\
\midrule
star & $0.00302$ &$0.00047$ &$0.00622$ &$0.00179$ &$0.04991$ &$0.00529$ &$0.01112$ \\
ext & $0.00999$ &$0.00041$ &$0.01544$ &$0.00769$ &$0.06152$ &$0.01639$ &$0.03165$ \\
per & $0.00848$ &\textbf{0.00045} &$0.00477$ &$0.00180$ &$0.06369$ &$0.00787$ &$0.01165$ \\
ctr & $0.01263$ &$0.00091$ &$0.01392$ &$0.00102$ &$0.05383$ &$0.01100$ &$0.03717$ \\
sym & $0.01299$ &$0.00047$ &$0.04411$ &$0.00947$ &$0.01876$ &$0.01391$ &$0.05245$ \\
asd &\textbf{0.00343} &$0.00046$ &$0.00570$ &\textbf{0.00135} &\textbf{0.03851} &$0.00339$ &\textbf{0.00830} \\
mix &$0.00373$ &\textbf{0.00045} &\textbf{0.00507} &$0.00162$ &$0.04687$ &\textbf{0.00374} &$0.00773$ \\
\bottomrule
\end{tabular}
}
\caption{Cross-evaluation of non-weighted point sets for $d=5$ and $n=256$. }
\label{tab:nonweighted_d5_n256}
\end{table}

\begin{table}[t]
\centering
\resizebox{\columnwidth}{!}{
\begin{tabular}{cccccccc}
\toprule
Set & $D_2^{*}$& $D_2^{\ext}$& $D_2^{\per}$& $D_2^{\ctr}$& $D_2^{\sym}$ & $D_2^{\asd}$ &$D_2^{\mix}$ \\
\midrule
star,$\boldsymbol{\gamma}$& $0.00365$ &$0.00050$ &$0.00548$ &$0.00154$ &$0.05126$ &$0.00407$ &$0.00854$ \\
ext,$\boldsymbol{\gamma}$& $0.00909$ &$0.00052$ &$0.00541$ &$0.00205$ &$0.07275$ &$0.00801$ &$0.01185$ \\
per,$\boldsymbol{\gamma}$& $0.00821$ &$0.00049$ &$0.00513$ &$0.00192$ &$0.06910$ &$0.00834$ &$0.01198$ \\
ctr,$\boldsymbol{\gamma}$& $0.00381$ &$0.00050$ &$0.00530$ &$0.00146$ &$0.04914$ &$0.00392$ &$0.00837$ \\
sym,$\boldsymbol{\gamma}$& $0.00382$ &$0.00050$ &$0.00545$ &$0.00149$ &$0.04840$ &\textbf{0.00383} &\textbf{0.00829} \\
asd& \textbf{0.00343} &\textbf{0.00046} &$0.00570$ &\textbf{0.00135} &\textbf{0.03851} &$0.00339$ &$0.00830$ \\
mix,$\boldsymbol{\gamma}$& $0.00418$ &$0.00048$ &\textbf{0.00520} &$0.00166$ &$0.05162$ &$0.00430$ &$0.00829$ \\
\bottomrule
\end{tabular}
}
\caption{Cross-evaluation of weighted optimized point sets ($\bsgamma=\bsone$; the ASD set is the unweighted one) for $d=5$ and $n=256$, evaluated on the unweighted discrepancies. In bold is the best set in each column, excluding the set optimized for that criterion (weighted or not).}
\label{tab:weighted_d5_n256}
\end{table}

\subsubsection{Dimension 10}\label{sec:dim10}
Here we report on experiments with $d=10$ and $n\le 2^{10}$. Random points become better than repeated points when $n\ge2^d$ for extreme discrepancy and when $n\ge 3^d$ for centered discrepancy. This sample size just barely reaches the threshold for $D_2^\ext$ and does not reach it for $D_2^\ctr$. Nor does it reach the $n\ge3^d$ threshold of the twin-point pathology for $D_2^\sym$. The results for $d=10$ show some unexpected enormous improvements over Sobol' points. Those lead us to find pathologies beyond the repeated-point pathology.

For the unweighted measures, Table \ref{tab:res:10d_comp} shows the improvements that MPMC found over Sobol' points.  Table~\ref{tab:mpmc:10d} in the Appendix shows the numerical values.  Some of them equal the values from their pathological point sets (a similar observation can already be made in 8 dimensions, see Table~\ref{tab:mpmc:8d}). 

MPMC scored extremely large improvements over Sobol' points for $D_2^\ctr$.  Table~\ref{tab:mpmc:10d} shows that MPMC reached the pathological value $12^{-d/2}$. The point sets it found are tight clouds around the center $\bsc$, typically within $0.005$ of $1/2$ in every coordinate: a slightly spread-out version of the repeated-center pathology of Section~\ref{sec:ctr}. We view this as a pathology with $D_2^\ctr$ and not a fault of the Sobol' points.

The periodic discrepancy has no repeated-point pathology and none of the other pathologies we have identified here. There MPMC does not improve by much over the Sobol' points.

The symmetric discrepancy also has no repeated-point pathology, but there MPMC does greatly improve over Sobol' points. The symmetric discrepancy has a twin-point pathology described in Section~\ref{sec:sym}. MPMC did find pathological point sets. For $n\le128$ at $d=8$ and for all $n$ at $d=10$, its solutions are within 1\% of the pathological value $3^{-d/2}$ of $D_2^\sym$. All of their points lie on the boundary of $[0,1]^d$. For every $n\le128$ at $d=8$, and for $n\in\{16,32,64,512\}$ at $d=10$, the solutions consist entirely of vertices, with equally many even and odd vertices (Section~\ref{sec:sym}). For example, the solution for $d=10$ and $n=512$ uses just three distinct vertices, $256$ of the points being even vertices and $256$ odd ones. For $n=128$ and $n=256$ at $d=10$, about 95\% of the points form pairs that differ only in one coordinate, which is $0$ in one point and $1$ in the other. The solution for $d=10$ and $n=1024$ also attains the pathological value but is not fully described by either construction.

For $d=10$, $D_2^\ext$ has a repeated-point pathology for $n<1024$, repeating any point in $\partial[0,1]^d$. We see that MPMC found solutions that were much better than Sobol' points for small $n$, losing that advantage as $n$ approached $1024$. The solutions for $n\le512$ at $d=10$ and for $n\le128$ at $d=8$ have most of their points (at least 80\%, and usually all of them) within $0.05$ of a vertex $\bsv$ or of the opposite vertex $\bsone-\bsv$, a two-cluster version of the repeated boundary-point pathology.

The repeated-point pathology for $D_2^\mix$ is mild with the cutoff at $n\approx (6/5)^d$ or just over $6$ for $d=10$. We see consistent modest improvements over Sobol' points.

The repeated-point pathology for $D_2^*$ persists until $n\approx 57$. Yet we continue to see large improvements from MPMC over Sobol' points up to $n=1024$.  Sobol' points always start at the origin, via $\bsx_1=\bszero$. Placing even one point at the origin makes a meaningful contribution to $D_2^*$. From Warnock's formula~\eqref{eq:warnock} we compute that deleting $\bsx_1$ or moving it to $\bsone$ reduces $D^*_2$ to about $2/3$ of its value for $n=1024$ and $d=10$ with the direction numbers of \cite{joe:kuo:2008}. So there is a lot of room for improvement. MPMC reduced $D_2^*$ to 48\% of the Sobol' point value.

The MPMC point sets for $D_2^*$, $D_2^\per$, $D_2^\asd$ and $D_2^\mix$, and for all of the weighted criteria, show no sign of a pathology at $d=8$ or $d=10$: their $n$ points are distinct and spread over $[0,1]^d$.

\begin{table}[t]
\centering
\begin{tabular}{rccccccc}
\toprule
$n$ & $D_2^{*}$& $D_2^{\ext}$& $D_2^{\per}$& $D_2^{\ctr}$& $D_2^{\sym}$&$D_2^{\asd}$&$D_2^{\mix}$ \\
\midrule
16& $0.05$ & $0.05$ & $0.98$ & $0.00205$ & $0.01$ & $0.93$ & $0.83$\\
32& $0.09$ & $0.11$ & $0.97$ & $0.00409$ & $0.02$ & $0.93$ & $0.88$\\
64& $0.14$ & $0.20$ & $0.97$ & $0.00813$ & $0.03$ & $0.90$ & $0.90$\\
128& $0.21$ & $0.31$ & $0.97$ & $0.0157$ & $0.04$ & $0.89$ & $0.92$\\
256& $0.30$ & $0.47$ & $0.94$ & $0.0298$ & $0.06$ & $0.85$ & $0.88$\\
512& $0.40$ & $0.70$ & $0.94$ & $0.0545$ & $0.09$ & $0.83$ & $0.87$\\
1024& $0.48$ & $0.84$ & $0.93$ & $0.0949$ & $0.13$ & $0.84$ & $0.90$\\
\bottomrule
\end{tabular}
 \caption{Same approach as in Table~\ref{tab:res:2d_comp}, this time in dimension 10. }
\label{tab:res:10d_comp}
\end{table}

For the weighted criterion in Table \ref{tab:mpmc:10d_compwei}, the star, centered and symmetric remain the measures for which we are able to consistently improve on the Sobol' sequence, while extreme and periodic remain the hardest to improve on. 
In our experiments, MPMC improved on the Sobol' points for almost all of the weighted and unweighted criteria and sample sizes considered here. The exceptions are the weighted extreme and periodic discrepancies at $d=2$ (Table~\ref{tab:mpmc:2d_compwei}).
The dip in the ratios at $n=256$ reflects relatively poor Sobol' points at that sample size rather than unusually good MPMC points. It is important to note that these are relative improvements: the results for periodic could simply indicate that the Sobol' sequence is particularly good for that measure initially, while it could also be poorer for centered for example.

\begin{table}[t]
\centering
\begin{tabular}{rccccccc}
\toprule
$n$ & $D_2^{*,\boldsymbol{\gamma}}$& $D_2^{\ext,\boldsymbol{\gamma}}$& $D_2^{\per,\boldsymbol{\gamma}}$& $D_2^{\ctr,\boldsymbol{\gamma}}$& $D_2^{\sym,\boldsymbol{\gamma}}$& $D_2^{\asd,\boldsymbol{\gamma}}$ &$D_2^{\mix,\boldsymbol{\gamma}}$ \\
\midrule
16& $0.41$ & $0.87$ & $0.92$ & $0.59$ & $0.82$ & $0.81$ & $0.87$\\
32& $0.52$ & $0.86$ & $0.88$ & $0.70$ & $0.84$ & $0.83$ & $0.86$\\
64& $0.60$ & $0.87$ & $0.90$ & $0.76$ & $0.83$ & $0.82$ & $0.87$\\
128& $0.69$ & $0.88$ & $0.91$ & $0.83$ & $0.86$ & $0.85$ & $0.90$\\
256& $0.68$ & $0.78$ & $0.82$ & $0.76$ & $0.79$ & $0.78$ & $0.81$\\
512& $0.71$ & $0.86$ & $0.87$ & $0.78$ & $0.78$ & $0.78$ & $0.84$\\
1024& $0.77$ & $0.93$ & $0.90$ & $0.82$ & $0.83$ & $0.82$ & $0.89$\\
\bottomrule
\end{tabular}
 \caption{Same approach as in Table~\ref{tab:mpmc:2d_compwei}, this time in dimension 10, with $\bsgamma=\bsone$.  }
\label{tab:mpmc:10d_compwei}
\end{table}

\subsection{Gradient descent approach}\label{sec:graddesc}

The main disadvantage with the MPMC approach is its computational cost. While it is feasible to run the models for $n=16$ on a laptop in a reasonable time frame, the $n=1024$ ones take many hours even without hyperparameter tuning. For practitioners with little time or computer power, a possible alternative is the gradient descent method in~\cite{FCL2Grad}. It re-uses the Warnock style formula to optimize directly the point placement via gradient descent (either via Adam as in the original paper, or a simple hand-tuned step size both work). The main advantages are its very low cost, around a minute for the largest instances in this paper, and the fact it rarely makes the point set worse.

However, it is very strongly dependent on the starting configuration. It is extremely unlikely to get good results starting from a random configuration (unlike MPMC), and if the point set has already been optimized, it will usually bring very marginal improvements at best. For traditional, asymptotic constructions like the Sobol' sequence, Table~\ref{tab:res:graddesc} shows that it always improved the different discrepancy values. 

At $n=1024$, MPMC gave larger improvements than gradient descent for every unweighted criterion and dimension. The comparison would only get worse with more hyperparameter tuning, while the gradient descent algorithm is limited to clever choices of the initial samples. This approach also struggles much more with the weighted discrepancy measures past dimension 2, as described in Table~\ref{tab:res:graddesc_wei}, even making the point set worse for one of the instances. Even so, for the weighted criteria gradient descent did better than MPMC in a few low-dimensional cases: $D_2^{*,\bsgamma}$, $D_2^{\ext,\bsgamma}$, $D_2^{\per,\bsgamma}$ and $D_2^{\sym,\bsgamma}$ at $d=2$, $D_2^{*,\bsgamma}$ at $d=3$, and $D_2^{\ctr,\bsgamma}$ at $d=4$.

\begin{table}[t]
\centering
\begin{tabular}{rccccc}
\toprule
$d$ & $D_2^{*}$& $D_2^{\ext}$& $D_2^{\per}$& $D_2^{\ctr}$& $D_2^{\sym}$ \\
\midrule
2& 0.71 & 0.90 & 0.98 & 0.82 & 0.78\\
3& 0.82 & 0.88 & 0.96 & 0.82 & 0.72\\
4& 0.84 & 0.85 & 0.94 & 0.79 & 0.65\\
5& 0.86 & 0.86 & 0.96 & 0.71 & 0.57\\
8& 0.74 & 0.90 & 0.98 & 0.38 & 0.43\\
10& 0.60 & 0.85 & 0.98 & 0.15 & 0.63\\
\bottomrule
\end{tabular}
 \caption{Discrepancy ratios for the gradient descent algorithms compared to the Sobol' sequence. All tests were done for $n=1024$, for varying dimensions and measures. }
\label{tab:res:graddesc}
\end{table}

\begin{table}[t]
\centering
\begin{tabular}{rccccc}
\toprule
$d$ & $D_2^{*,\boldsymbol{\gamma}}$& $D_2^{\ext,\boldsymbol{\gamma}}$& $D_2^{\per,\boldsymbol{\gamma}}$& $D_2^{\ctr,\boldsymbol{\gamma}}$& $D_2^{\sym,\boldsymbol{\gamma}}$ \\
\midrule
2& 0.67 & 0.83 & 0.96 & 0.81 & 0.65\\
3& 0.76 & 1.08 & 0.99 & 0.84 & 0.88\\
4& 0.84 & 0.96 & 0.98 & 0.88 & 0.93\\
5& 0.89 & 1.00 & 0.99 & 0.94 & 0.94\\
8& 0.94 & 0.98 & 0.99 & 0.97 & 0.97\\
10& 0.94 & 0.99 & 0.99 & 0.97 & 0.97\\
\bottomrule
\end{tabular}
 \caption{Weighted discrepancy ratios for the gradient descent algorithms compared to the Sobol' sequence. All tests were done for $n=1024$, for varying dimensions and measures.}
\label{tab:res:graddesc_wei}
\end{table}

\subsection{Greedy Approaches}\label{sec:greedy}
We implement models for the direct optimization of low-discrepancy sequences with respect to the extreme, average squared, and centered $L_2$ discrepancies. Our aim is to constructively come up with sequences by adding one point at a time. This approach is designed for settings with small budgets. A practitioner might start with $20$ very expensive runs and then add new points sequentially while retaining the prior ones.
For the chosen measures, the approach is greedy: given an initial point set, each new point is selected to minimize the overall discrepancy when added. All experiments in this section were conducted using Gurobi~\cite{gurobi}, with a 120-second time limit imposed on each optimization instance.

Our construction is inspired by the strong empirical performance of the Kritzinger sequence~\cite{krit:2022,clem:2023} for the $L_2$ star discrepancy in dimension 1. We adopt a similar sequential strategy, but generalize the objective to alternative $L_2$ discrepancies of interest.

We do not tackle weighted discrepancies here, quite simply because it is already a hard problem to optimize for the discrepancies without lower-dimensional projections. Our methods are not yet satisfactory for weighted discrepancies, though we point out that it is possible to construct such sequences by moving only a single point by gradient descent. This remains very experimental and there are open questions on how to even initialize the placement of this new point for each gradient descent.

Another advantage of $L_2$-based discrepancy measures, not mentioned earlier, is that the contribution of a new point also admits a closed-form expression. For example, the contribution of a new point $\bsy \in [0,1]^d$ to the squared $L_2$ star discrepancy of points 
$\bsx_1, \dots, \bsx_n$ is given by
\begin{multline}\label{eq:addonept}
F^*(\bsy; \bsx_1, \ldots, \bsx_n) := -2^{1-d} \prod_{k=1}^d (1 - y_k^2) + \frac{1}{n+1} \prod_{k=1}^d (1 - y_k) \\
+ \frac{2}{n+1} \sum_{i=1}^n \prod_{k=1}^d (1 - \max(y_k, x_{i,k})),
\end{multline}
which can be evaluated in $O(nd)$ time. Here the contribution of $\bsy$ is $n+1$ times the part of the new squared discrepancy that depends on $\bsy$:
\[
(D_2^*(\bsx_1,\dots,\bsx_n,\bsy))^2=C_n+\frac{F^*(\bsy;\bsx_1,\ldots,\bsx_n)}{n+1},
\]
where
\[
C_n=\frac1{3^d}-\frac{2}{n+1}\sum_{i=1}^n\prod_{k=1}^d\frac{1-x_{i,k}^2}{2}+\frac1{(n+1)^2}\sum_{i,j=1}^n\prod_{k=1}^d\bigl(1-\max(x_{i,k},x_{j,k})\bigr)
\]
does not depend on $\bsy$.  This enables efficient greedy selection of the next point via
\begin{equation*}
\bsx_{n+1} = \arg\min_{\bsy \in [0,1]^d} F^*(\bsy; \bsx_1, \ldots, \bsx_n).
\end{equation*}
We extend this idea to other $L_2$ discrepancy variants, for which similar expressions for $F^{\bullet}(\bsy; \bsx_1, \ldots, \bsx_n)$ can often be derived. For further details, we refer the reader to~\cite{clem:2023,clethesis24,krit:2022}.

Unlike the $L_2$ star discrepancy, the extreme, average squared, and centered $L_2$ discrepancies perform poorly when the sequence is initialized from a single-point set. In particular, for both the average squared and centered discrepancies, initializing at $(0.5, 0.5)$ leads to degenerate behavior: the optimal next point is repeatedly chosen as $(0.5, 0.5)$, resulting in a set containing only duplicated points.

To overcome this, we instead add multiple points at each step. While this approach alleviates the degeneracy, it comes with significantly higher computational cost. Figure \ref{fig:+4Gr} illustrates this strategy, showing the average squared and centered discrepancies when starting from the singleton set $\{(0.5, 0.5)\}$ and adding 4 points at a time---a natural choice given the symmetries in the definitions of these measures. The eventual rise in discrepancy values seen in Figure \ref{fig:+4Gr} is not due to a failure of the method, but rather a practical limitation: the optimizer reaches the 120-second time limit before it can identify a valid batch of 4 candidate points. This issue could be mitigated by relaxing the time constraint if needed.

\begin{figure}[t]
    \centering
    \includegraphics[width=0.49\linewidth]{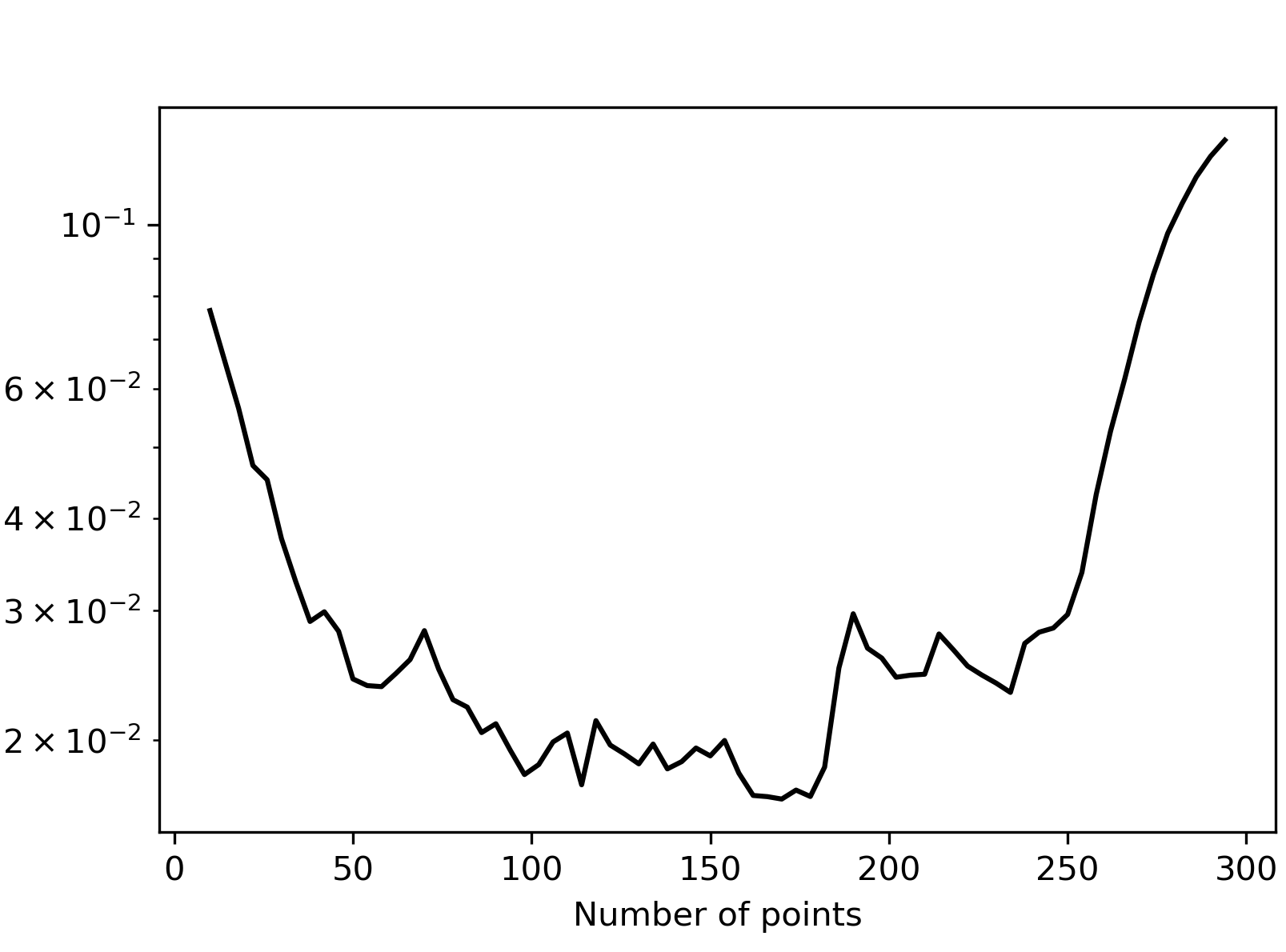}
\hfill
    \includegraphics[width=0.49\linewidth]{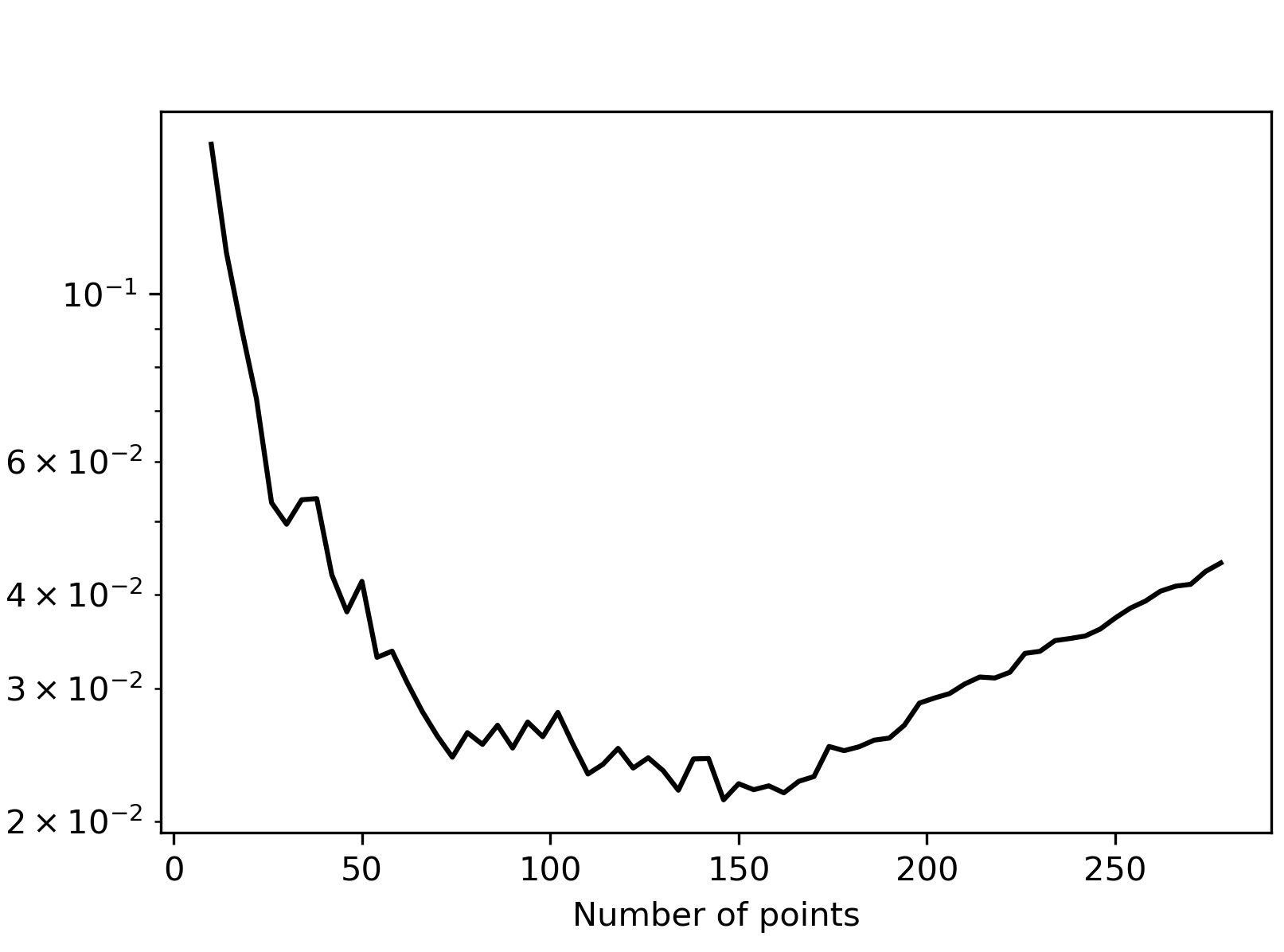}
    \caption{Average squared (left) and centered $L_2$ discrepancies (right) obtained when greedily adding 4 points at a time, starting with the singleton $\{(0.5,0.5)\}$. The time limit is set at 120s for every 4 points to add, leading to the rise in discrepancy past a certain point.}
    \label{fig:+4Gr}
\end{figure}

Performance of the greedy selection algorithm improves noticeably when it is initialized from a high-quality point set---in this case, an MPMC point set optimized for the same discrepancy. Figure \ref{fig:+1} (left) shows the evolution of the average squared $L_2$ discrepancy as one point at a time is added to a $d = 2$, $n = 128$ MPMC set. While some improvement is observed, the final discrepancy remains higher than it could be largely due to a poor first step. Moreover, the resulting trajectory lacks the stability exhibited by the Kritzinger sequence. In contrast, the $L_2$ extreme discrepancy shown in Figure \ref{fig:+1} (right) displays behavior much closer to that of the Kritzinger sequence---stable across iterations and consistently maintaining low discrepancy values.

\begin{figure}[t]
    \centering
    \includegraphics[width=0.48\linewidth]{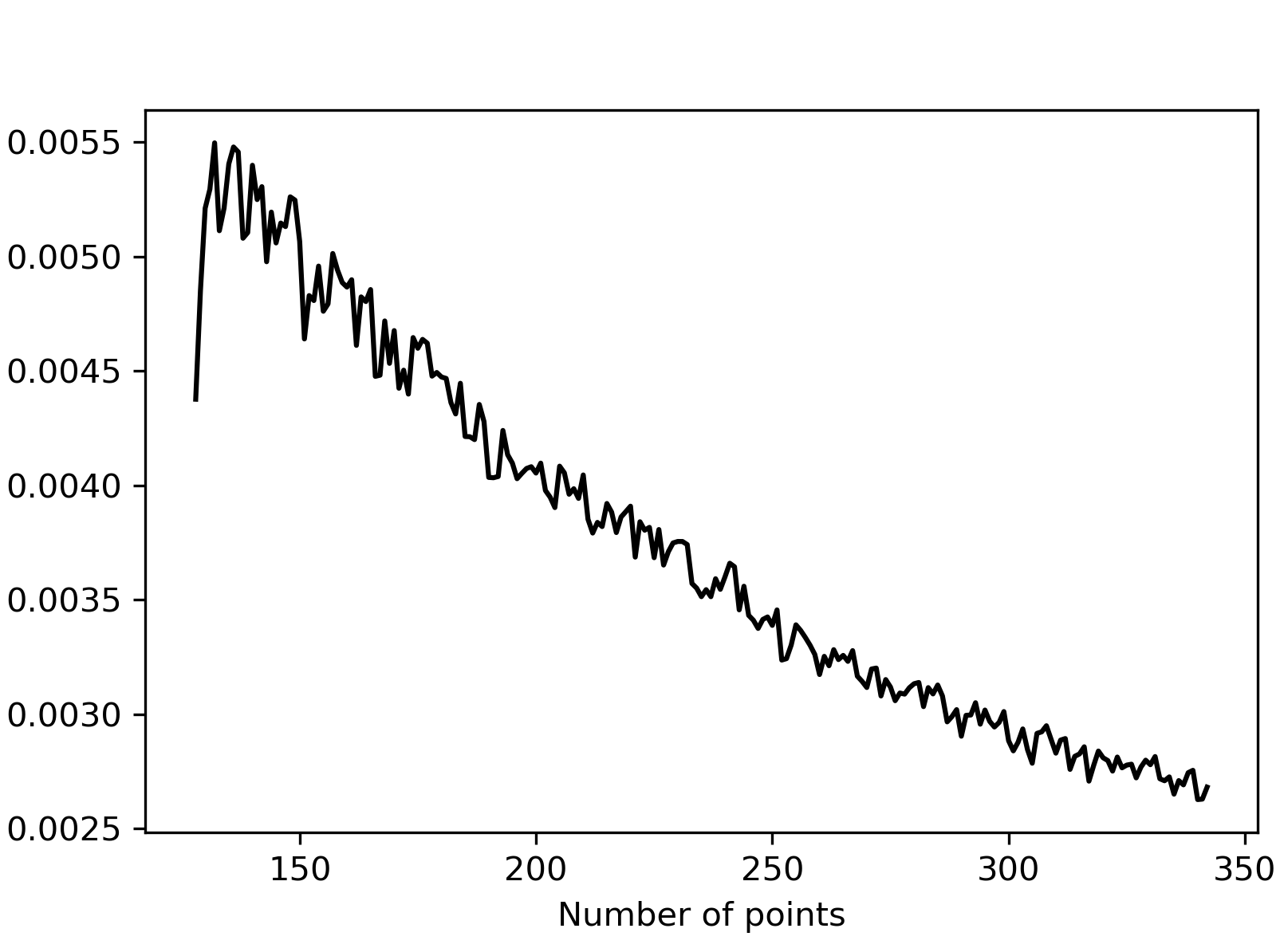}
    \hfill
    \includegraphics[width=0.48\linewidth]{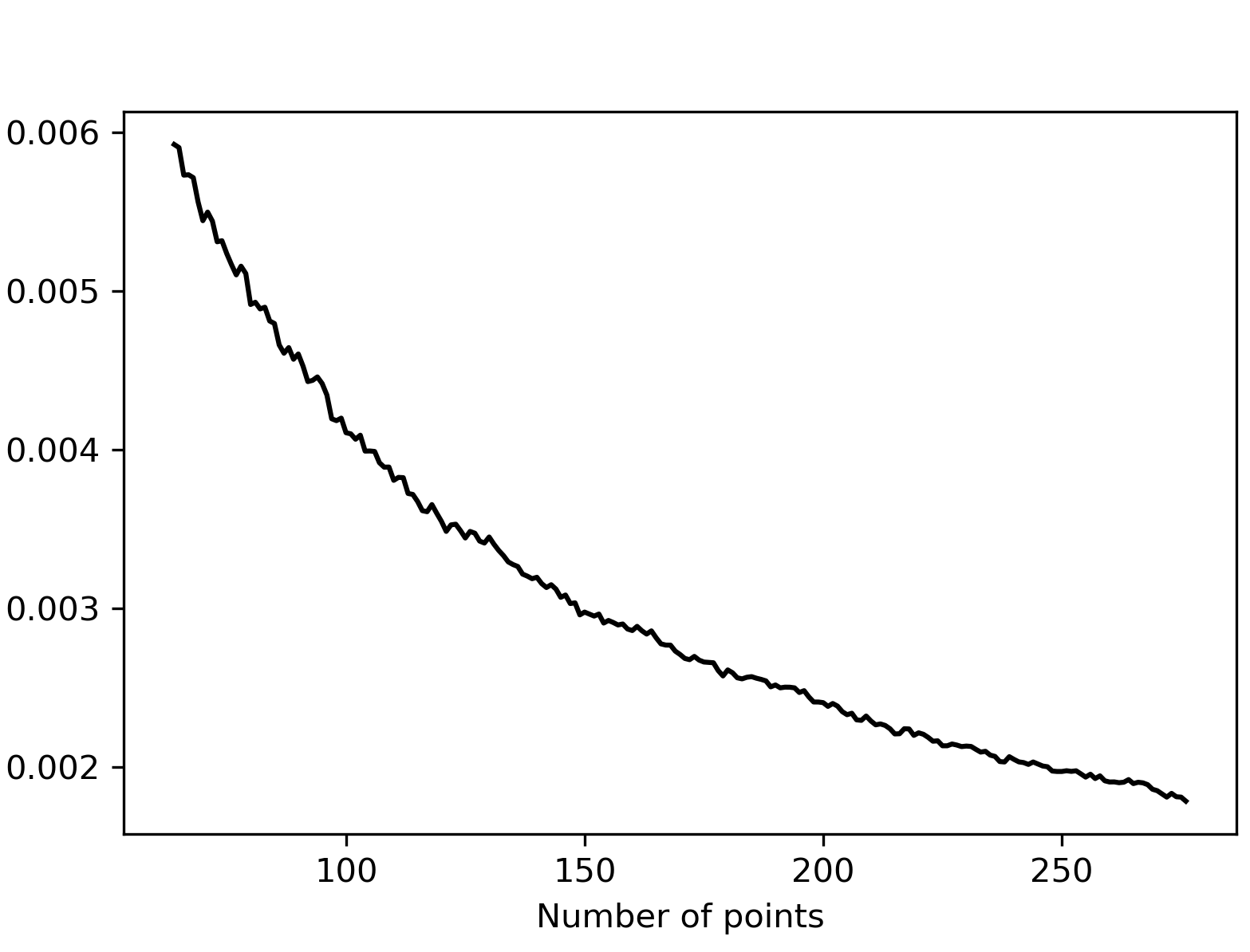}
    \caption{Average squared (left) and extreme (right) $L_2$ discrepancies obtained when greedily adding the single optimal point for the respective measures. Starting with MPMC set sizes of respectively 128 and 64 points. }
    \label{fig:+1}
\end{figure}

\section{Discussion}\label{sec:discussion}

We have studied a collection of $L_2$ discrepancy measures designed for optimizations.  All of them have Warnock-style formulas allowing them to be computed in $O(dn^2)$ work. Optimization requires a usable gradient, a criterion satisfied by all except $D_2^\cad$. Only $D_2^*$ treats the origin differently from other vertices but removing that asymmetry does not necessarily remove repeated-point pathologies. All fully $d$-dimensional criteria except $D^\sym_2$ and $D_2^\per$ have repeated-point pathologies, while $D_2^\sym$ has a severe twin-point pathology.

Weighted versions of these $L_2$ discrepancies have less severe pathologies than fully $d$-dimensional versions do.  The ASD is equivalent to a weighted version of $D_2^\sym$ and it is subject to a recursive twin-point pathology.
We found no pathology for $D_2^\per$.

Our numerical results show that point set optimization through the discrepancy can be nontrivial in practice depending on the exact measure and the optimization method. 

In our experiments, MPMC found point sets with small discrepancy for every criterion we tried, and greedy constructions extended good MPMC starting sets one point at a time. Whether this holds for other dimensions, sample sizes and criteria remains open.

We do not know why ASD was better in cross-evaluations than $D_2^\per$. Two considerations come to mind.  First, the pathology in ASD is not present at the important sample sizes for $d\le 10$.  Second, $D_2^\per$ is a periodic criterion which identifies $0$ and $1$ while all of the other criteria are non-periodic. 

We do not recommend using $D_2^*$. It privileges the origin and is subject to a severe pathology. In our examples, optimizing the ASD produced $D_2^*$ values nearly as good as those from optimizing $D_2^*$ and we saw points designed to optimize $D_2^*$ do poorly on other criteria.

We hope this work encourages further interest in the optimization of discrepancy measures—both in the development of effective optimization methods and in the thoughtful selection of appropriate objective functions, including the study of their pathologies.

\paragraph{Acknowledgments} NK was supported by NSF grant DMS-2316011 while undertaking this work.
ABO was supported by NSF grant DMS-2152780. TKR is supported by the Hector Foundation.
We thank Fred J. Hickernell for several helpful conversations throughout the development of this manuscript. NK also acknowledges Lijia Lin and Ally Pascual Kwan for help running numerical experiments during their time as 2025 undergraduate summer research students at Illinois Tech supported by NSF grant DMS-2244553.

\bibliographystyle{plain}
\bibliography{qmc2,additionalbib2,NMK25_2}

\newpage
\appendix
\begin{center}
{\bf Supplementary Material for:}\\
On the optimization of discrepancy measures
\end{center}

\section{Additional Empirical Results}\label{app:empirical}

We give the exact discrepancy values for all experiments contained in Section~\ref{sec:numerical}. All values are $D_2^\bullet$ (not squared), evaluated in double precision from the point sets and truncated to three significant figures; see the start of Section~\ref{sec:mpmc}.

\begin{table}[h]
\tiny
\centering
\begin{tabular}{cccccccc}
\toprule
$n$ & $D_2^{*}$& $D_2^{\ext}$& $D_2^{\per}$& $D_2^{\ctr}$& $D_2^{\sym}$&$D_2^{\asd}$&$D_2^{\mix}$ \\
\midrule
16& $2.50\times 10^{-2}$ & $1.56\times 10^{-2}$ & $3.65\times 10^{-2}$ & $2.12\times 10^{-2}$ & $7.00\times 10^{-2}$ & $2.62\times 10^{-2}$ & $3.70\times 10^{-2}$\\
32& $1.33\times 10^{-2}$ & $8.73\times 10^{-3}$ & $2.04\times 10^{-2}$ & $1.17\times 10^{-2}$ & $3.91\times 10^{-2}$ & $1.40\times 10^{-2}$ & $2.02\times 10^{-2}$\\
64& $7.31\times 10^{-3}$ & $4.81\times 10^{-3}$ & $1.11\times 10^{-2}$ & $6.34\times 10^{-3}$ & $2.27\times 10^{-2}$ & $7.62\times 10^{-3}$ & $1.07\times 10^{-2}$\\
128& $3.90\times 10^{-3}$ & $2.64\times 10^{-3}$ & $6.04\times 10^{-3}$ & $3.45\times 10^{-3}$ & $1.25\times 10^{-2}$ & $4.02\times 10^{-3}$ & $5.81\times 10^{-3}$\\
256& $2.10\times 10^{-3}$ & $1.46\times 10^{-3}$ & $3.26\times 10^{-3}$ & $1.84\times 10^{-3}$ & $6.72\times 10^{-3}$ & $2.15\times 10^{-3}$ & $3.07\times 10^{-3}$\\
512& $1.10\times 10^{-3}$ & $7.89\times 10^{-4}$ & $1.71\times 10^{-3}$ & $9.88\times 10^{-4}$ & $3.63\times 10^{-3}$ & $1.15\times 10^{-3}$ & $1.70\times 10^{-3}$\\
1024& $5.70\times 10^{-4}$ & $4.24\times 10^{-4}$ & $1.03\times 10^{-3}$ & $5.30\times 10^{-4}$ & $1.88\times 10^{-3}$ & $6.28\times 10^{-4}$ & $9.01\times 10^{-4}$\\
\bottomrule
\end{tabular}
 \caption{Unweighted discrepancy values obtained using MPMC in dimension 2.}
\label{tab:mpmc:2d}
\end{table}

\begin{table}[h]
\tiny
\centering
\begin{tabular}{cccccccc}
\toprule
$n$ & $D_2^{*}$& $D_2^{\ext}$& $D_2^{\per}$& $D_2^{\ctr}$& $D_2^{\sym}$&$D_2^{\asd}$&$D_2^{\mix}$ \\
\midrule
16& $2.72\times 10^{-2}$ & $9.65\times 10^{-3}$ & $3.92\times 10^{-2}$ & $1.53\times 10^{-2}$ & $9.00\times 10^{-2}$ & $2.91\times 10^{-2}$ & $4.88\times 10^{-2}$\\
32& $1.55\times 10^{-2}$ & $5.81\times 10^{-3}$ & $2.33\times 10^{-2}$ & $9.93\times 10^{-3}$ & $5.58\times 10^{-2}$ & $1.68\times 10^{-2}$ & $2.77\times 10^{-2}$\\
64& $8.68\times 10^{-3}$ & $3.46\times 10^{-3}$ & $1.35\times 10^{-2}$ & $5.96\times 10^{-3}$ & $3.48\times 10^{-2}$ & $9.40\times 10^{-3}$ & $1.56\times 10^{-2}$\\
128& $4.84\times 10^{-3}$ & $2.03\times 10^{-3}$ & $7.78\times 10^{-3}$ & $3.46\times 10^{-3}$ & $2.12\times 10^{-2}$ & $5.29\times 10^{-3}$ & $8.73\times 10^{-3}$\\
256& $2.73\times 10^{-3}$ & $1.17\times 10^{-3}$ & $4.50\times 10^{-3}$ & $1.96\times 10^{-3}$ & $1.27\times 10^{-2}$ & $2.93\times 10^{-3}$ & $4.80\times 10^{-3}$\\
512& $1.49\times 10^{-3}$ & $6.84\times 10^{-4}$ & $2.51\times 10^{-3}$ & $1.11\times 10^{-3}$ & $7.27\times 10^{-3}$ & $1.60\times 10^{-3}$ & $2.58\times 10^{-3}$\\
1024& $7.96\times 10^{-4}$ & $3.83\times 10^{-4}$ & $1.40\times 10^{-3}$ & $6.16\times 10^{-4}$ & $4.23\times 10^{-3}$ & $8.86\times 10^{-4}$ & $1.56\times 10^{-3}$\\
\bottomrule
\end{tabular}
 \caption{Unweighted discrepancy values obtained using MPMC in dimension 3.}
\label{tab:mpmc:3d}
\end{table}

\begin{table}[h]
\tiny
\centering
\begin{tabular}{cccccccc}
\toprule
$n$ & $D_2^{*}$& $D_2^{\ext}$& $D_2^{\per}$& $D_2^{\ctr}$& $D_2^{\sym}$&$D_2^{\asd}$&$D_2^{\mix}$ \\
\midrule
16& $2.54\times 10^{-2}$ & $4.82\times 10^{-3}$ & $3.62\times 10^{-2}$ & $6.75\times 10^{-3}$ & $8.74\times 10^{-2}$ & $2.84\times 10^{-2}$ & $5.54\times 10^{-2}$\\
32& $1.54\times 10^{-2}$ & $3.13\times 10^{-3}$ & $2.23\times 10^{-2}$ & $5.51\times 10^{-3}$ & $6.13\times 10^{-2}$ & $1.68\times 10^{-2}$ & $3.29\times 10^{-2}$\\
64& $9.08\times 10^{-3}$ & $1.98\times 10^{-3}$ & $1.36\times 10^{-2}$ & $3.96\times 10^{-3}$ & $4.09\times 10^{-2}$ & $9.92\times 10^{-3}$ & $1.94\times 10^{-2}$\\
128& $5.27\times 10^{-3}$ & $1.22\times 10^{-3}$ & $8.21\times 10^{-3}$ & $2.60\times 10^{-3}$ & $2.71\times 10^{-2}$ & $5.81\times 10^{-3}$ & $1.13\times 10^{-2}$\\
256& $3.08\times 10^{-3}$ & $7.51\times 10^{-4}$ & $4.93\times 10^{-3}$ & $1.61\times 10^{-3}$ & $1.66\times 10^{-2}$ & $3.38\times 10^{-3}$ & $6.53\times 10^{-3}$\\
512& $1.71\times 10^{-3}$ & $4.59\times 10^{-4}$ & $2.95\times 10^{-3}$ & $9.69\times 10^{-4}$ & $1.09\times 10^{-2}$ & $1.94\times 10^{-3}$ & $3.67\times 10^{-3}$\\
1024& $9.63\times 10^{-4}$ & $2.75\times 10^{-4}$ & $1.69\times 10^{-3}$ & $5.74\times 10^{-4}$ & $7.21\times 10^{-3}$ & $1.05\times 10^{-3}$ & $2.06\times 10^{-3}$\\
\bottomrule
\end{tabular}
 \caption{Unweighted discrepancy values obtained using MPMC in dimension 4.}
\label{tab:mpmc:4d}
\end{table}

\begin{table}[h]
\tiny
\centering
\begin{tabular}{cccccccc}
\toprule
$n$ & $D_2^{*}$& $D_2^{\ext}$& $D_2^{\per}$& $D_2^{\ctr}$& $D_2^{\sym}$&$D_2^{\asd}$&$D_2^{\mix}$ \\
\midrule
16& $2.15\times 10^{-2}$ & $2.00\times 10^{-3}$ & $3.01\times 10^{-2}$ & $2.00\times 10^{-3}$ & $6.41\times 10^{-2}$ & $2.49\times 10^{-2}$ & $5.94\times 10^{-2}$\\
32& $1.36\times 10^{-2}$ & $1.46\times 10^{-3}$ & $1.93\times 10^{-2}$ & $1.98\times 10^{-3}$ & $5.35\times 10^{-2}$ & $1.53\times 10^{-2}$ & $3.62\times 10^{-2}$\\
64& $8.35\times 10^{-3}$ & $9.85\times 10^{-4}$ & $1.21\times 10^{-2}$ & $1.80\times 10^{-3}$ & $3.99\times 10^{-2}$ & $9.38\times 10^{-3}$ & $2.20\times 10^{-2}$\\
128& $5.07\times 10^{-3}$ & $6.41\times 10^{-4}$ & $7.65\times 10^{-3}$ & $1.42\times 10^{-3}$ & $2.77\times 10^{-2}$ & $5.69\times 10^{-3}$ & $1.31\times 10^{-2}$\\
256& $3.02\times 10^{-3}$ & $4.11\times 10^{-4}$ & $4.77\times 10^{-3}$ & $1.02\times 10^{-3}$ & $1.87\times 10^{-2}$ & $3.39\times 10^{-3}$ & $7.73\times 10^{-3}$\\
512& $1.78\times 10^{-3}$ & $2.60\times 10^{-4}$ & $2.92\times 10^{-3}$ & $6.79\times 10^{-4}$ & $1.23\times 10^{-2}$ & $2.05\times 10^{-3}$ & $4.53\times 10^{-3}$\\
1024& $1.01\times 10^{-3}$ & $1.60\times 10^{-4}$ & $1.78\times 10^{-3}$ & $4.28\times 10^{-4}$ & $8.95\times 10^{-3}$ & $1.15\times 10^{-3}$ & $2.75\times 10^{-3}$\\
\bottomrule
\end{tabular}
 \caption{Unweighted discrepancy values obtained using MPMC in dimension 5.}
\label{tab:mpmc:5d}
\end{table}

\begin{table}[h]
\tiny
\centering
\begin{tabular}{cccccccc}
\toprule
$n$ & $D_2^{*}$& $D_2^{\ext}$& $D_2^{\per}$& $D_2^{\ctr}$& $D_2^{\sym}$&$D_2^{\asd}$&$D_2^{\mix}$ \\
\midrule
16& $8.43\times 10^{-3}$ & $4.82\times 10^{-5}$ & $1.36\times 10^{-2}$ & $4.82\times 10^{-5}$ & $1.23\times 10^{-2}$ & $1.25\times 10^{-2}$ & $5.42\times 10^{-2}$\\
32& $6.31\times 10^{-3}$ & $4.82\times 10^{-5}$ & $9.24\times 10^{-3}$ & $4.82\times 10^{-5}$ & $1.23\times 10^{-2}$ & $8.28\times 10^{-3}$ & $3.55\times 10^{-2}$\\
64& $4.45\times 10^{-3}$ & $4.82\times 10^{-5}$ & $6.22\times 10^{-3}$ & $4.82\times 10^{-5}$ & $1.23\times 10^{-2}$ & $5.40\times 10^{-3}$ & $2.28\times 10^{-2}$\\
128& $3.01\times 10^{-3}$ & $4.82\times 10^{-5}$ & $4.17\times 10^{-3}$ & $4.82\times 10^{-5}$ & $1.23\times 10^{-2}$ & $3.50\times 10^{-3}$ & $1.47\times 10^{-2}$\\
256& $1.96\times 10^{-3}$ & $3.83\times 10^{-5}$ & $2.77\times 10^{-3}$ & $4.82\times 10^{-5}$ & $1.15\times 10^{-2}$ & $2.26\times 10^{-3}$ & $9.35\times 10^{-3}$\\
512& $1.26\times 10^{-3}$ & $2.74\times 10^{-5}$ & $1.81\times 10^{-3}$ & $4.79\times 10^{-5}$ & $9.73\times 10^{-3}$ & $1.46\times 10^{-3}$ & $5.92\times 10^{-3}$\\
1024& $8.10\times 10^{-4}$ & $1.88\times 10^{-5}$ & $1.20\times 10^{-3}$ & $4.58\times 10^{-5}$ & $7.57\times 10^{-3}$ & $9.37\times 10^{-4}$ & $3.63\times 10^{-3}$\\
\bottomrule
\end{tabular}
 \caption{Unweighted discrepancy values obtained using MPMC in dimension 8.}
\label{tab:mpmc:8d}
\end{table}

\begin{table}[h]
\tiny
\centering
\begin{tabular}{cccccccc}
\toprule
$n$ & $D_2^{*}$& $D_2^{\ext}$& $D_2^{\per}$& $D_2^{\ctr}$& $D_2^{\sym}$&$D_2^{\asd}$&$D_2^{\mix}$ \\
\midrule
16& $3.51\times 10^{-3}$ & $4.01\times 10^{-6}$ & $7.32\times 10^{-3}$ & $4.01\times 10^{-6}$ & $4.11\times 10^{-3}$ & $6.91\times 10^{-3}$ & $4.47\times 10^{-2}$\\
32& $2.93\times 10^{-3}$ & $4.01\times 10^{-6}$ & $5.04\times 10^{-3}$ & $4.01\times 10^{-6}$ & $4.11\times 10^{-3}$ & $4.69\times 10^{-3}$ & $3.06\times 10^{-2}$\\
64& $2.28\times 10^{-3}$ & $4.01\times 10^{-6}$ & $3.46\times 10^{-3}$ & $4.01\times 10^{-6}$ & $4.11\times 10^{-3}$ & $3.16\times 10^{-3}$ & $2.04\times 10^{-2}$\\
128& $1.67\times 10^{-3}$ & $4.01\times 10^{-6}$ & $2.37\times 10^{-3}$ & $4.01\times 10^{-6}$ & $4.14\times 10^{-3}$ & $2.12\times 10^{-3}$ & $1.35\times 10^{-2}$\\
256& $1.17\times 10^{-3}$ & $4.01\times 10^{-6}$ & $1.61\times 10^{-3}$ & $4.01\times 10^{-6}$ & $4.11\times 10^{-3}$ & $1.40\times 10^{-3}$ & $8.88\times 10^{-3}$\\
512& $7.99\times 10^{-4}$ & $4.01\times 10^{-6}$ & $1.09\times 10^{-3}$ & $4.01\times 10^{-6}$ & $4.11\times 10^{-3}$ & $9.31\times 10^{-4}$ & $5.72\times 10^{-3}$\\
1024& $5.27\times 10^{-4}$ & $3.29\times 10^{-6}$ & $7.36\times 10^{-4}$ & $4.01\times 10^{-6}$ & $4.11\times 10^{-3}$ & $6.14\times 10^{-4}$ & $3.73\times 10^{-3}$\\
\bottomrule
\end{tabular}
 \caption{Unweighted discrepancy values obtained using MPMC in dimension 10.}
\label{tab:mpmc:10d}
\end{table}

\begin{table}[h]
\tiny
\centering
\begin{tabular}{cccccccc}
\toprule
$n$ & $D_2^{*,\boldsymbol{\gamma}}$& $D_2^{\ext,\boldsymbol{\gamma}}$& $D_2^{\per,\boldsymbol{\gamma}}$& $D_2^{\ctr,\boldsymbol{\gamma}}$& $D_2^{\sym,\boldsymbol{\gamma}}$& $D_2^{\asd,\boldsymbol{\gamma}}$ &$D_2^{\mix,\boldsymbol{\gamma}}$ \\
\midrule
16& $3.81\times 10^{-2}$ & $3.15\times 10^{-2}$ & $5.30\times 10^{-2}$ & $3.59\times 10^{-2}$ & $3.54\times 10^{-2}$ & $3.77\times 10^{-2}$ & $4.99\times 10^{-2}$\\
32& $2.03\times 10^{-2}$ & $1.66\times 10^{-2}$ & $2.86\times 10^{-2}$ & $1.96\times 10^{-2}$ & $1.89\times 10^{-2}$ & $2.04\times 10^{-2}$ & $2.65\times 10^{-2}$\\
64& $1.07\times 10^{-2}$ & $8.97\times 10^{-3}$ & $1.56\times 10^{-2}$ & $1.00\times 10^{-2}$ & $1.01\times 10^{-2}$ & $1.06\times 10^{-2}$ & $1.41\times 10^{-2}$\\
128& $5.67\times 10^{-3}$ & $4.73\times 10^{-3}$ & $8.16\times 10^{-3}$ & $5.33\times 10^{-3}$ & $5.32\times 10^{-3}$ & $5.68\times 10^{-3}$ & $7.45\times 10^{-3}$\\
256& $2.92\times 10^{-3}$ & $2.53\times 10^{-3}$ & $4.35\times 10^{-3}$ & $2.78\times 10^{-3}$ & $2.82\times 10^{-3}$ & $2.93\times 10^{-3}$ & $3.94\times 10^{-3}$\\
512& $1.53\times 10^{-3}$ & $1.36\times 10^{-3}$ & $2.28\times 10^{-3}$ & $1.43\times 10^{-3}$ & $1.45\times 10^{-3}$ & $1.50\times 10^{-3}$ & $1.98\times 10^{-3}$\\
1024& $8.79\times 10^{-4}$ & $7.22\times 10^{-4}$ & $1.29\times 10^{-3}$ & $8.07\times 10^{-4}$ & $8.78\times 10^{-4}$ & $8.76\times 10^{-4}$ & $1.08\times 10^{-3}$\\
\bottomrule
\end{tabular}
 \caption{Weighted discrepancy values obtained using MPMC in dimension 2, with $\bsgamma=\bsone$.}
\label{tab:mpmc:2d_wei}
\end{table}

\begin{table}[h]
\tiny
\centering
\begin{tabular}{cccccccc}
\toprule
$n$ & $D_2^{*,\boldsymbol{\gamma}}$& $D_2^{\ext,\boldsymbol{\gamma}}$& $D_2^{\per,\boldsymbol{\gamma}}$& $D_2^{\ctr,\boldsymbol{\gamma}}$& $D_2^{\sym,\boldsymbol{\gamma}}$& $D_2^{\asd,\boldsymbol{\gamma}}$ &$D_2^{\mix,\boldsymbol{\gamma}}$ \\
\midrule
16& $6.89\times 10^{-2}$ & $4.75\times 10^{-2}$ & $9.35\times 10^{-2}$ & $6.03\times 10^{-2}$ & $5.95\times 10^{-2}$ & $6.90\times 10^{-2}$ & $9.48\times 10^{-2}$\\
32& $3.75\times 10^{-2}$ & $2.65\times 10^{-2}$ & $5.29\times 10^{-2}$ & $3.32\times 10^{-2}$ & $3.28\times 10^{-2}$ & $3.82\times 10^{-2}$ & $5.37\times 10^{-2}$\\
64& $2.04\times 10^{-2}$ & $1.41\times 10^{-2}$ & $2.93\times 10^{-2}$ & $1.81\times 10^{-2}$ & $1.79\times 10^{-2}$ & $2.06\times 10^{-2}$ & $2.85\times 10^{-2}$\\
128& $1.10\times 10^{-2}$ & $7.83\times 10^{-3}$ & $1.59\times 10^{-2}$ & $9.81\times 10^{-3}$ & $9.60\times 10^{-3}$ & $1.10\times 10^{-2}$ & $1.56\times 10^{-2}$\\
256& $5.82\times 10^{-3}$ & $4.18\times 10^{-3}$ & $8.45\times 10^{-3}$ & $5.26\times 10^{-3}$ & $5.18\times 10^{-3}$ & $5.83\times 10^{-3}$ & $8.22\times 10^{-3}$\\
512& $3.01\times 10^{-3}$ & $2.18\times 10^{-3}$ & $4.58\times 10^{-3}$ & $2.77\times 10^{-3}$ & $2.76\times 10^{-3}$ & $3.07\times 10^{-3}$ & $4.51\times 10^{-3}$\\
1024& $1.69\times 10^{-3}$ & $1.17\times 10^{-3}$ & $2.63\times 10^{-3}$ & $1.48\times 10^{-3}$ & $1.47\times 10^{-3}$ & $1.75\times 10^{-3}$ & $2.33\times 10^{-3}$\\
\bottomrule
\end{tabular}
 \caption{Weighted discrepancy values obtained using MPMC in dimension 3, with $\bsgamma=\bsone$.}
\label{tab:mpmc:3d_wei}
\end{table}

\begin{table}[h]
\tiny
\centering
\begin{tabular}{cccccccc}
\toprule
$n$ & $D_2^{*,\boldsymbol{\gamma}}$& $D_2^{\ext,\boldsymbol{\gamma}}$& $D_2^{\per,\boldsymbol{\gamma}}$& $D_2^{\ctr,\boldsymbol{\gamma}}$& $D_2^{\sym,\boldsymbol{\gamma}}$& $D_2^{\asd,\boldsymbol{\gamma}}$ &$D_2^{\mix,\boldsymbol{\gamma}}$ \\
\midrule
16& $1.09\times 10^{-1}$ & $6.48\times 10^{-2}$ & $1.47\times 10^{-1}$ & $9.14\times 10^{-2}$ & $9.23\times 10^{-2}$ & $1.13\times 10^{-1}$ & $1.63\times 10^{-1}$\\
32& $6.32\times 10^{-2}$ & $3.73\times 10^{-2}$ & $8.61\times 10^{-2}$ & $5.12\times 10^{-2}$ & $5.16\times 10^{-2}$ & $6.42\times 10^{-2}$ & $9.41\times 10^{-2}$\\
64& $3.50\times 10^{-2}$ & $2.08\times 10^{-2}$ & $4.87\times 10^{-2}$ & $2.90\times 10^{-2}$ & $2.86\times 10^{-2}$ & $3.58\times 10^{-2}$ & $5.21\times 10^{-2}$\\
128& $1.92\times 10^{-2}$ & $1.15\times 10^{-2}$ & $2.73\times 10^{-2}$ & $1.62\times 10^{-2}$ & $1.59\times 10^{-2}$ & $1.92\times 10^{-2}$ & $2.94\times 10^{-2}$\\
256& $1.04\times 10^{-2}$ & $6.40\times 10^{-3}$ & $1.52\times 10^{-2}$ & $8.79\times 10^{-3}$ & $8.72\times 10^{-3}$ & $1.04\times 10^{-2}$ & $1.59\times 10^{-2}$\\
512& $5.62\times 10^{-3}$ & $3.40\times 10^{-3}$ & $8.66\times 10^{-3}$ & $4.71\times 10^{-3}$ & $4.61\times 10^{-3}$ & $5.59\times 10^{-3}$ & $8.59\times 10^{-3}$\\
1024& $3.08\times 10^{-3}$ & $2.02\times 10^{-3}$ & $5.08\times 10^{-3}$ & $2.82\times 10^{-3}$ & $2.61\times 10^{-3}$ & $3.05\times 10^{-3}$ & $4.91\times 10^{-3}$\\
\bottomrule
\end{tabular}
 \caption{Weighted discrepancy values obtained using MPMC in dimension 4, with $\bsgamma=\bsone$.}
\label{tab:mpmc:4d_wei}
\end{table}

\begin{table}[h]
\tiny
\centering
\begin{tabular}{cccccccc}
\toprule
$n$ & $D_2^{*,\boldsymbol{\gamma}}$& $D_2^{\ext,\boldsymbol{\gamma}}$& $D_2^{\per,\boldsymbol{\gamma}}$& $D_2^{\ctr,\boldsymbol{\gamma}}$& $D_2^{\sym,\boldsymbol{\gamma}}$& $D_2^{\asd,\boldsymbol{\gamma}}$ &$D_2^{\mix,\boldsymbol{\gamma}}$ \\
\midrule
16& $1.70\times 10^{-1}$ & $8.36\times 10^{-2}$ & $2.22\times 10^{-1}$ & $1.24\times 10^{-1}$ & $1.25\times 10^{-1}$ & $1.74\times 10^{-1}$ & $2.64\times 10^{-1}$\\
32& $9.83\times 10^{-2}$ & $4.91\times 10^{-2}$ & $1.30\times 10^{-1}$ & $7.37\times 10^{-2}$ & $7.34\times 10^{-2}$ & $1.00\times 10^{-1}$ & $1.55\times 10^{-1}$\\
64& $5.60\times 10^{-2}$ & $2.91\times 10^{-2}$ & $7.64\times 10^{-2}$ & $4.21\times 10^{-2}$ & $4.25\times 10^{-2}$ & $5.67\times 10^{-2}$ & $9.10\times 10^{-2}$\\
128& $3.11\times 10^{-2}$ & $1.61\times 10^{-2}$ & $4.45\times 10^{-2}$ & $2.45\times 10^{-2}$ & $2.39\times 10^{-2}$ & $3.23\times 10^{-2}$ & $4.94\times 10^{-2}$\\
256& $1.72\times 10^{-2}$ & $8.89\times 10^{-3}$ & $2.52\times 10^{-2}$ & $1.32\times 10^{-2}$ & $1.34\times 10^{-2}$ & $1.75\times 10^{-2}$ & $2.81\times 10^{-2}$\\
512& $9.42\times 10^{-3}$ & $4.94\times 10^{-3}$ & $1.45\times 10^{-2}$ & $7.62\times 10^{-3}$ & $7.50\times 10^{-3}$ & $9.74\times 10^{-3}$ & $1.57\times 10^{-2}$\\
1024& $5.42\times 10^{-3}$ & $2.82\times 10^{-3}$ & $8.49\times 10^{-3}$ & $4.25\times 10^{-3}$ & $4.21\times 10^{-3}$ & $5.42\times 10^{-3}$ & $8.90\times 10^{-3}$\\
\bottomrule
\end{tabular}
 \caption{Weighted discrepancy values obtained using MPMC in dimension 5, with $\bsgamma=\bsone$.}
\label{tab:mpmc:5d_wei}
\end{table}

\begin{table}[h]
\tiny
\centering
\begin{tabular}{cccccccc}
\toprule
$n$ & $D_2^{*,\boldsymbol{\gamma}}$& $D_2^{\ext,\boldsymbol{\gamma}}$& $D_2^{\per,\boldsymbol{\gamma}}$& $D_2^{\ctr,\boldsymbol{\gamma}}$& $D_2^{\sym,\boldsymbol{\gamma}}$& $D_2^{\asd,\boldsymbol{\gamma}}$ &$D_2^{\mix,\boldsymbol{\gamma}}$ \\
\midrule
16& $4.78\times 10^{-1}$ & $1.59\times 10^{-1}$ & $5.90\times 10^{-1}$ & $2.58\times 10^{-1}$ & $2.74\times 10^{-1}$ & $5.03\times 10^{-1}$ & $9.29\times 10^{-1}$\\
32& $2.93\times 10^{-1}$ & $9.70\times 10^{-2}$ & $3.67\times 10^{-1}$ & $1.65\times 10^{-1}$ & $1.68\times 10^{-1}$ & $3.04\times 10^{-1}$ & $5.72\times 10^{-1}$\\
64& $1.76\times 10^{-1}$ & $5.73\times 10^{-2}$ & $2.27\times 10^{-1}$ & $1.02\times 10^{-1}$ & $1.02\times 10^{-1}$ & $1.80\times 10^{-1}$ & $3.44\times 10^{-1}$\\
128& $1.05\times 10^{-1}$ & $3.48\times 10^{-2}$ & $1.40\times 10^{-1}$ & $6.23\times 10^{-2}$ & $6.21\times 10^{-2}$ & $1.07\times 10^{-1}$ & $2.03\times 10^{-1}$\\
256& $6.06\times 10^{-2}$ & $1.97\times 10^{-2}$ & $8.35\times 10^{-2}$ & $3.73\times 10^{-2}$ & $3.71\times 10^{-2}$ & $6.32\times 10^{-2}$ & $1.20\times 10^{-1}$\\
512& $3.57\times 10^{-2}$ & $1.18\times 10^{-2}$ & $5.07\times 10^{-2}$ & $2.19\times 10^{-2}$ & $2.17\times 10^{-2}$ & $3.67\times 10^{-2}$ & $7.14\times 10^{-2}$\\
1024& $2.08\times 10^{-2}$ & $7.04\times 10^{-3}$ & $3.13\times 10^{-2}$ & $1.32\times 10^{-2}$ & $1.30\times 10^{-2}$ & $2.18\times 10^{-2}$ & $4.34\times 10^{-2}$\\
\bottomrule
\end{tabular}
 \caption{Weighted discrepancy values obtained using MPMC in dimension 8, with $\bsgamma=\bsone$.}
\label{tab:mpmc:8d_wei}
\end{table}

\begin{table}[h]
\tiny
\centering
\begin{tabular}{cccccccc}
\toprule
$n$ & $D_2^{*,\boldsymbol{\gamma}}$& $D_2^{\ext,\boldsymbol{\gamma}}$& $D_2^{\per,\boldsymbol{\gamma}}$& $D_2^{\ctr,\boldsymbol{\gamma}}$& $D_2^{\sym,\boldsymbol{\gamma}}$& $D_2^{\asd,\boldsymbol{\gamma}}$ &$D_2^{\mix,\boldsymbol{\gamma}}$ \\
\midrule
16& $8.46\times 10^{-1}$ & $2.18\times 10^{-1}$ & $1.05\times 10^{0}$ & $3.75\times 10^{-1}$ & $4.12\times 10^{-1}$ & $9.08\times 10^{-1}$ & $1.95\times 10^{0}$\\
32& $5.30\times 10^{-1}$ & $1.37\times 10^{-1}$ & $6.67\times 10^{-1}$ & $2.50\times 10^{-1}$ & $2.60\times 10^{-1}$ & $5.71\times 10^{-1}$ & $1.21\times 10^{0}$\\
64& $3.36\times 10^{-1}$ & $8.40\times 10^{-2}$ & $4.26\times 10^{-1}$ & $1.58\times 10^{-1}$ & $1.63\times 10^{-1}$ & $3.48\times 10^{-1}$ & $7.61\times 10^{-1}$\\
128& $2.06\times 10^{-1}$ & $5.09\times 10^{-2}$ & $2.66\times 10^{-1}$ & $1.01\times 10^{-1}$ & $1.02\times 10^{-1}$ & $2.13\times 10^{-1}$ & $4.65\times 10^{-1}$\\
256& $1.23\times 10^{-1}$ & $3.04\times 10^{-2}$ & $1.63\times 10^{-1}$ & $6.17\times 10^{-2}$ & $6.26\times 10^{-2}$ & $1.29\times 10^{-1}$ & $2.80\times 10^{-1}$\\
512& $7.41\times 10^{-2}$ & $1.84\times 10^{-2}$ & $1.01\times 10^{-1}$ & $3.79\times 10^{-2}$ & $3.78\times 10^{-2}$ & $7.75\times 10^{-2}$ & $1.70\times 10^{-1}$\\
1024& $4.51\times 10^{-2}$ & $1.16\times 10^{-2}$ & $6.36\times 10^{-2}$ & $2.33\times 10^{-2}$ & $2.35\times 10^{-2}$ & $4.67\times 10^{-2}$ & $1.05\times 10^{-1}$\\
\bottomrule
\end{tabular}
 \caption{Weighted discrepancy values obtained using MPMC in dimension 10, with $\bsgamma=\bsone$.}
\label{tab:mpmc:10d_wei}
\end{table}

\end{document}